\documentclass[a4paper]{amsart}
\usepackage{amsmath}
\usepackage{amsfonts}
\usepackage{amssymb}
\usepackage{amsthm} 
\usepackage{graphicx}
\usepackage{color}
\usepackage[top = 2cm, left = 3cm, right = 3cm, bottom = 2cm]{geometry}

\usepackage{hyperref}

\usepackage[backend=bibtex,citestyle =numeric-comp,bibstyle = numeric, maxnames= 12,firstinits=true,abbreviate=true,isbn = false, url = false,doi=false,eprint=false,date=year]{biblatex}
\AtEveryBibitem{%
\clearfield{pages}
\clearfield{month}
\clearfield{month}
\clearfield{note}
\clearfield{number}
\clearlist{language}
}
\renewbibmacro{in:}{}

\newtheorem{definition}{Definition}[section]

\newtheorem{theorem}[definition]{Theorem}

\newtheorem{remark}[definition]{Remark}
\newtheorem{claim}[definition]{Claim}

\numberwithin{equation}{section}

\newcommand{\init}{\mathrm{init}}
\newcommand{\sign}{\mathrm{sign}\,}
\newcommand{\bk}{\mathbf{k}}
\newcommand{\br}{\mathbf{r}}
\newcommand{\bchi}{\boldsymbol\chi}

\title[Mathematical  modeling of self-generated gradients]{Mathematical  modeling of cell  collective motion triggered by  self-generated gradients}

\author[V. Calvez]{Vincent Calvez}
\address{Institut Camille Jordan (ICJ), UMR 5208 CNRS \& Universit\'e Claude Bernard Lyon~1, and Equipe-Projet Inria Dracula, Lyon, France}
\email{vincent.calvez@math.cnrs.fr}

\author[M. Demircigil]{Mete Demircigil}
\address{Institut Camille Jordan (ICJ), UMR 5208 CNRS \& Universit\'e Claude Bernard Lyon~1, and Equipe-Projet Inria Dracula, Lyon, France}
\email{mete.demircigil@univ-lyon1.fr}

\author[R. Sublet]{Roxana Sublet}
\address{Institut Camille Jordan (ICJ), UMR 5208 CNRS \& Universit\'e Claude Bernard Lyon~1, Lyon, France}
\email{roxana.sublet23@gmail.com}

\begin{document}
 
\begin{abstract}
Self-generated gradients have atttracted a lot of attention in the recent biological literature. It is considered as a robust strategy for a group of cells to find its way during a long journey. 
This note is intended to discuss various scenarios for modeling traveling waves of cells that constantly deplete a chemical cue, and so create their own signaling gradient all along the way. We begin with one famous model by Keller and Segel for bacterial chemotaxis. We present the model and the construction of the traveling wave solutions.  We also discuss the limitation of this approach, and review some subsequent work addressing stability issues. Next, we review  two relevant extensions, which are supported by biological experiments. They both admit traveling wave solutions with an explicit value for the wave speed. We conclude by discussing some open problems and perspectives, and  particularly a striking  mechanism of speed determinacy occurring at the back of the wave. All the results presented in this note are illustrated by numerical simulations.  
\end{abstract}

\maketitle

\section{Introduction}

It has been now 50 years that Evelyn F. Keller and Lee A. Segel published their article "{\em Traveling bands of chemotactic bacteria: A theoretical analysis}" \cite{keller_traveling_1971}, which is part of a series of works about the modeling of chemotaxis in bacteria {\em Esherichia coli} and amoebae {\em Dictyostelium discoideum} (shortnamed as Dicty in the following)  \cite{keller_conflict_1970, keller_initiation_1970, keller_model_1971, keller_traveling_1971}. This article described in a simple and elegant way the propagation of chemotactic waves of {\em E. coli} in a one-dimensional space, echoing the remarkable experiments by Adler performed in a capillary tube \cite{adler_chemotaxis_1966}. 

In the present contribution, the seminal ideas of Keller and Segel are discussed from a modern perspective, after half a century of intense activity at the interface of mathematics and biology. Our goal is not to review exhaustively various directions of research in the modeling of  chemotaxis. Our narrow objective consists in setting the focus on the notion of {\em self-generated gradient} (SGG), which has recently shed a new light on several biological processes, both in bacterial collective motion, and in some aspects of developmental biology \cite{tweedy_self-generated_2016-1,tweedy_self-generated_2020}. SGG are at the heart of the model in \cite{keller_traveling_1971}, in which cells create their own signaling gradient by consuming some nutrient, while moving collectively from one side of the domain to the other. There, collective motion results from the averaged biases in the  individual trajectories,  in response to nutrient heterogeneities, a process called chemotaxis.   This concept of SGG can be generalized to any situation where the signal depletion and chemotaxis functions overlap within the same cells \cite{scherber_epithelial_2012,tweedy_self-generated_2016,tweedy_self-generated_2016-1}.

\subsubsection*{SGG in waves of bacteria.}
The work of Keller and Segel has initiated a wealth of studies on bacterial chemotaxis. We refer to the comprehensive review of Tindall et al \cite{tindall_overview_2008}, and also the recent studies \cite{fu_spatial_2018,cremer_chemotaxis_2019} for new biological questions in this topic.  Most of the works discussed in this note consider short time experiments, or experiments at low level of nutrients, neglecting the effect of cell division. This makes a clear distinction between SGG and reaction-diffusion waves, as the celebrated Fisher/Kolmogorov-Petrovsky-Piskunov (F/KPP) equation \cite{fisher_wave_1937,kolmogorov_etude_1937,
aronson_multidimensional_1978}. For this reason, we shall not comment further about the numerous modeling contributions following the patterns reported by Budrene and Berg \cite{budrene_complex_1991,
budrene_dynamics_1995,
brenner_physical_1998} (ring expansion followed by formation of bacteria spots with remarkable symmetries). Chemotaxis has been shown to be crucial in the emergence of such patterns. However, the dynamics of ring expansion are mainly driven by growth and diffusion such as described by F/KPP,  (but see \cite{cremer_chemotaxis_2019} for a recent study where chemotaxis has been shown to enhance range expansion).

There exist many modeling contributions of chemotaxis in bacteria \cite{tindall_overview_2008,hillen_users_2009}, with a particular emphasis on the derivation of macroscopic models from individual rules through kinetic transport equations, see {\em e.g.} \cite{alt_biased_1980,
othmer_models_1988,
othmer_diffusion_2002,chalub_kinetic_2004,
filbet_derivation_2005,
chalub_model_2006,
bellomo_toward_2015}. In contrast, the number of contributions about  mathematical analysis of traveling waves without growth beyond  \cite{keller_traveling_1971}  is relatively scarce.  We refer to \cite{horstmann_constructive_2004}, for an (algebraic) extension of \cite{keller_traveling_1971} with more general chemotaxis functions and uptake rates. We also refer to the series of articles by Z.A. Wang and co-authors, see \cite{wang_mathematics_2012} for a preliminary review and below for further discussion.  

\subsubsection*{SGG in development and cancer.}
In developmental biology, cell movement over long distances is mediated by navigating cues, including chemotactic factors \cite{majumdar_new_2014}. It is commonly postulated that external, pre-patterned gradients, drive cellular migration. One of the key conceptual advantage of SGG is to free the developmental process from the requirement of pre-imposed long-range chemoattractant gradients. In contrast, SGG travel together with the cells, so that they can experience similar environmental conditions (chemical concentration, gradient steepness) all over the journey. This is thought to provide robustness to the developmental system \cite{tweedy_self-generated_2016,tweedy_self-generated_2020}. 

Recently, SGG have been shown to occur during embryogenesis, and in particular during the initiation of the posterior lateral line in zebrafish \cite{dona_directional_2013, venkiteswaran_generation_2013}. More precisely, migrating cell cohorts (consisting of approximately a hundred of cells) can generate and sustain gradients of chemoattractants across their length.  
This experimental work is of great importance as being the first proof of the occurrence of SGG {\em in vivo}. 

Self-generated gradients are also under investigation during cancer invasion and metastasis. This includes modeling {\em in silico} (see \cite{sfakianakis_hybrid_2020} and  references therein), and experiments with cell cultures {\em in vitro} \cite{muinonen-martin_melanoma_2014}. In particular, we highlight the work of \cite{scherber_epithelial_2012}, in which an astonishing self-guidance strategy in cancer epithelial cell populations was unravelled. In fact, cells were put in microfluidic mazes, without any pre-existing external gradients. Most of them could find their way  out of the mazes by generating their own navigating cues. Experimental studies with increasingly complex mazes were also performed with Dicty cells, with quite remarkable outcomes \cite{tweedy_seeing_2020}.

%
%

\subsubsection*{Plan and purpose of the paper.} In Section \ref{sec:KS-construct} we recall the basic construction of traveling waves in the seminal article \cite{keller_traveling_1971}. The lack of positivity of the chemical concentration is illustrated by some numerical simulations (Section \ref{sec:posistab}). The issue of instability is also reviewed. Section \ref{sec:variations} briefly presents some possible variations of the original article from the literature. It is one of the main goal of the present contribution to discuss in details two possible extensions which are biologically relevant (that is, supported by experiments). Section \ref{sec:scenario 1} contains an overview of past work where another attractant signal is added to prevent cell dispersion during propagation. This results in competing cell fluxes, with stronger advection at the back of the wave than at the edge. Section \ref{sec:scenario 2} reports on a piece of recent work including signal-dependent phenotypical switch (division/ migration). This results in a wave sustained by cell division restricted to the edge. 

All mathematical results proven here are simple, namely involving explicit construction of one-dimensional traveling waves (whose respective stabilities are supported by numerical simulations of the Cauchy problems). The last construction is original, up to our knowledge, see Theorem \ref{th:roxana}. It could be of interest for experts in reaction-diffusion equations, as it exhibits a possibly new phenomenon of selection of the minimal speed at the back of the wave.  

\subsubsection*{Acknowledgement.} Part of this work has been achieved during the third author's master internship at Institut Camille Jordan. The two first authors are indebted to Christophe Anjard, Olivier Cochet-Escartin and Jean-Paul Rieu for having drawn their attention to SGG beyond the case of bacterial waves. The authors are very grateful to Eitan Tadmor for his constant encouragement to write this contribution, and to a pair of anonymous reviewers for their feedbacks. reviewers This project has received financial support from the CNRS through the MITI interdisciplinary programs. This project has received funding from the European Research Council (ERC) under the European Union’s Horizon 2020 research and innovation program (grant agreement No 865711).

\section{The Keller-Segel model and variations}

\subsection{The construction of waves by Keller and Segel}
\label{sec:KS-construct}

In this section, we recall briefly the model and analysis in \cite{keller_traveling_1971}. The cell density (bacteria) is denoted by $\rho(t,x)$, for time $t>0$, and position along the channel axis $x\in \mathbb{R}$, whereas the concentration of the signaling molecule is denoted by $S(t,x)$. 
\begin{equation}\label{eq:KS}
\begin{cases}
\dfrac{\partial \rho}{\partial t} +  \dfrac{\partial }{\partial x}\left ( - d \dfrac{\partial \rho}{\partial x} + \chi \rho \dfrac{\partial \log S}{\partial x}   \right )  =  0\,,
\medskip\\
\dfrac{\partial S}{\partial t} = D \dfrac{\partial^2 S}{\partial x^2} - k \rho\,.
\end{cases}
\end{equation}
The equation on $\rho$ combines unbiased (diffusive) motion with directed motion in response to the logarithmic signaling gradient (see below for further discussion about this specific choice), with intensity $\chi>0$. 

On the one hand, the equation on $\rho$ is conservative, and the total mass of cells in the channel, which is an invariant of the system, is denoted by $M$, so that $M = \int_\mathbb{R} \rho(0,z)\, dz = \int_\mathbb{R} \rho(t,z)\, dz$. On the other hand, the chemical concentration decays globally in time, and the limiting value at $\infty$ is denoted by $S_\init$, which can be viewed as the  initial, homogeneous, concentration in the channel associated with the Cauchy problem. 

Noticeably, the consumption term in the equation on $S$, namely $-k\rho$, does not involve $S$ itself, precluding any guarantee about the positivity of $S$ in the long time. Nevertheless, the existence of positive traveling wave solutions $\rho(x-ct)$, $S(x-ct)$ was established in \cite{keller_traveling_1971} by means of explicit computations, in the absence of signal diffusion $D=0$ (for mathematical purposes), and with the condition $\chi>d$. The wave under interest has the following structure: $\rho \in L_+^1(\mathbb{R})$, with $\lim_{z\to\pm\infty}\rho(z) = 0$, and $S \in L_+^\infty(\mathbb{R})$ is increasing with $\lim_{z\to-\infty}S(z) = 0$, and $\lim_{z\to+\infty}S(z) = S_\init$, the reference value of the chemical concentration.  

\begin{theorem}[Keller and Segel \cite{keller_traveling_1971}]
Assume $D = 0$, and $\chi>d$. Then, there exist a speed $c>0$ (depending on $M$, $k$ and $S_\init$, but not on $\chi$ nor on $d$), and a stationary solution of \eqref{eq:KS} in the moving frame $(\rho(x-ct), S(x-ct))$, such that $\rho$ is positive and integrable, $\int_\mathbb{R} \rho(z)\, dz = M$, and $S$ is increasing between the following limiting values
\begin{equation*}
\begin{cases}
\lim_{z\to-\infty}S(z) = 0\,, \quad\\
\lim_{z\to+\infty}S(z) = S_\init\,.
\end{cases}
\end{equation*}
\label{th:KS}
\end{theorem}

Before we recall briefly the construction of the wave solution, let us comment on the value of the wave speed $c$, that can be directly obtained from the second equation in \eqref{eq:KS-S}, whatever the value of $D\geq 0$ is. Indeed,  the equation in the moving frame reads
\begin{equation*}
- c \dfrac{d S}{dz} = D \dfrac{d^2 S}{dz^2} - k \rho \,.
\end{equation*}
By integrating this equation over the line, and using the extremal conditions at $\pm \infty$ (that can be verified {\em a posteriori}), we find
\begin{equation}\label{eq:speed KS}
c S_\init = k \int_\mathbb{R} \rho(z)\, dz  = k M\,.
\end{equation}
Strikingly, the wave speed depends only on the dynamics of establishment of the gradient. In particular, it does not depend on the intensity of the chemotactic response $\chi$. This is in contrast with several conclusions to be drawn from alternative models in the sequel (see Sections \ref{sec:scenario 1} and \ref{sec:scenario 2}).  

\begin{proof}
The speed $c$ is given {\em a priori} by the relationship \eqref{eq:speed KS}.

The first step in the construction of traveling wave solutions  is the zero-flux condition in the moving frame $z = x-ct$, namely
\begin{align*}
- c \rho  - d \dfrac{d \rho}{d z} + \chi \rho \dfrac{d \log S}{d z} = 0 \quad& \Leftrightarrow \quad \dfrac{d \log \rho}{d z} = - \frac{c}{d} + \frac{\chi}{d} \dfrac{d \log S}{d z}\medskip\\
& \Leftrightarrow \quad \rho(z) = a \exp\left ( - \frac{c}{d} z + \frac{\chi}{d}  \log S \right )\,,
\end{align*}
where $a$ is a (positive) constant of integration. 
The second step consists in solving the following ODE (assuming $D=0$) 
\begin{equation*}
c \dfrac{d S}{dz} =  k a \exp\left ( - \frac{c}{d} z + \frac{\chi}{d}  \log S \right ) \quad \Leftrightarrow\quad 
\left(1 - \frac{\chi}{d}\right )^{-1} \left ( S_\init^{1 - \frac{\chi}{d}} - S(z)^{1 - \frac{\chi}{d}} \right ) = \frac{kad}{c^2} \exp\left ( - \frac c d z \right )\,. 
\end{equation*}
By re-arranging the terms, we obtain
\begin{equation*}
\left ( \dfrac{S(z)}{S_\init} \right )^{1 - \frac{\chi}{d}} = 1 + \left( \frac{\chi}{d} - 1 \right )\left ( \frac{kad}{c^2} S_\init^{\frac{\chi}{d}-1} \right ) \exp\left ( - \frac c d z \right )\,. 
\end{equation*}
Suppose that $\chi<d$, then the right-hand-side goes to $-\infty$ as $z\to -\infty$ which is a contradiction. Hence, the calculations make sense only if $\chi>d$.  
By translational invariance, the constant $a$ can be chosen so as to cancel the prefactor in the right-hand-side (provided $\chi>d$), yielding the simple expression
\begin{equation}\label{eq:KS-rho}
 \dfrac{S(z)}{S_\init}  = \left (  1 +  \exp\left ( - \frac c d z \right )\right )^ {\frac{d}{d - \chi}}\,.
\end{equation} 
The corresponding density profile is: 
\begin{equation}\label{eq:KS-S}
\rho(z) = a' \exp\left ( - \frac{c}{d} z\right ) \left (  1 +  \exp\left ( - \frac c d z \right )\right )^ {\frac{\chi}{d - \chi}}\,,
\end{equation}
for some constant $a'$, that can be determined explicitly through the conservation of mass. 
\end{proof}

\subsection{Positivity and stability issues}
\label{sec:posistab}

\begin{figure}
\begin{center}
\includegraphics[width=.66\linewidth]{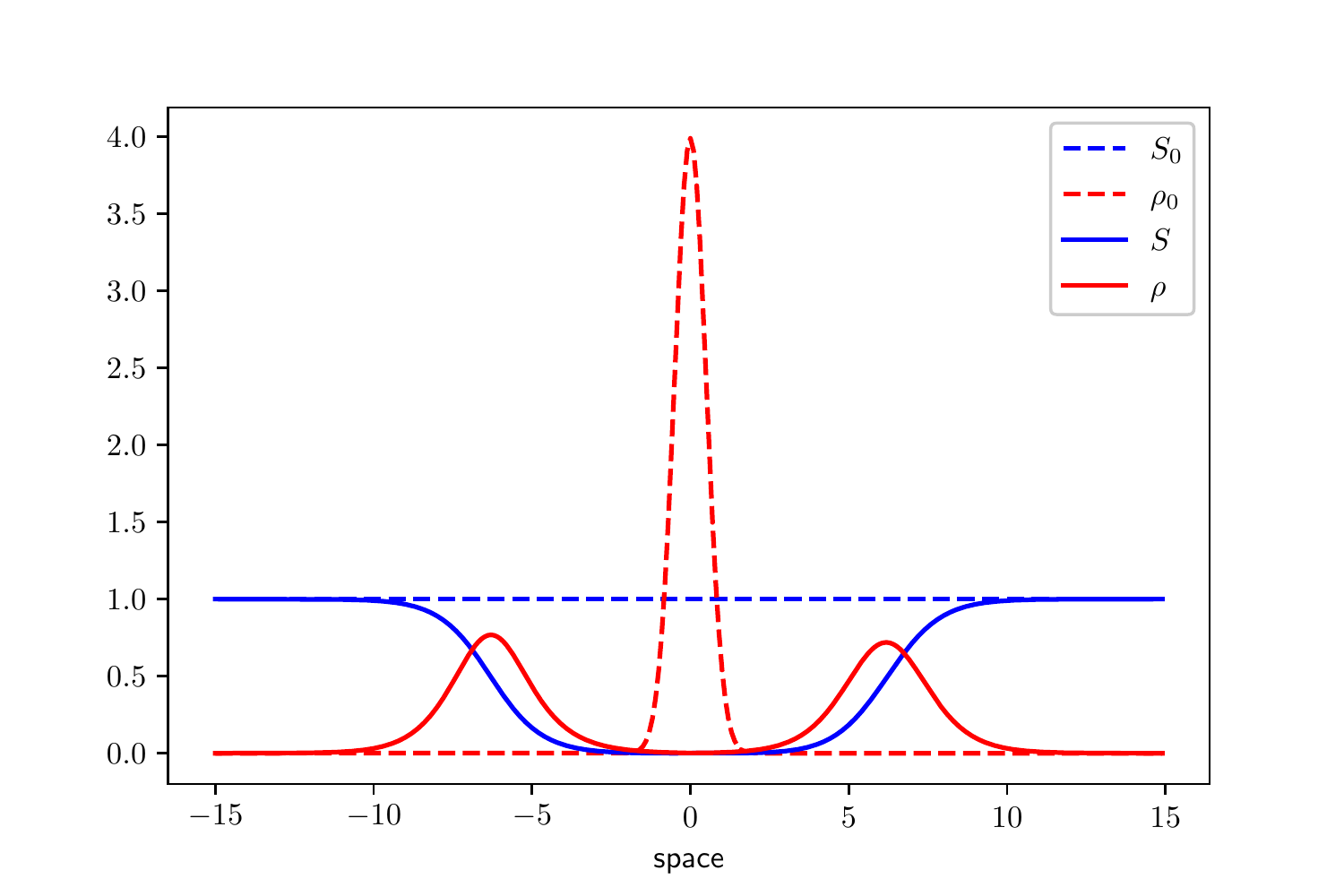}(a)\\
\includegraphics[width=.66\linewidth]{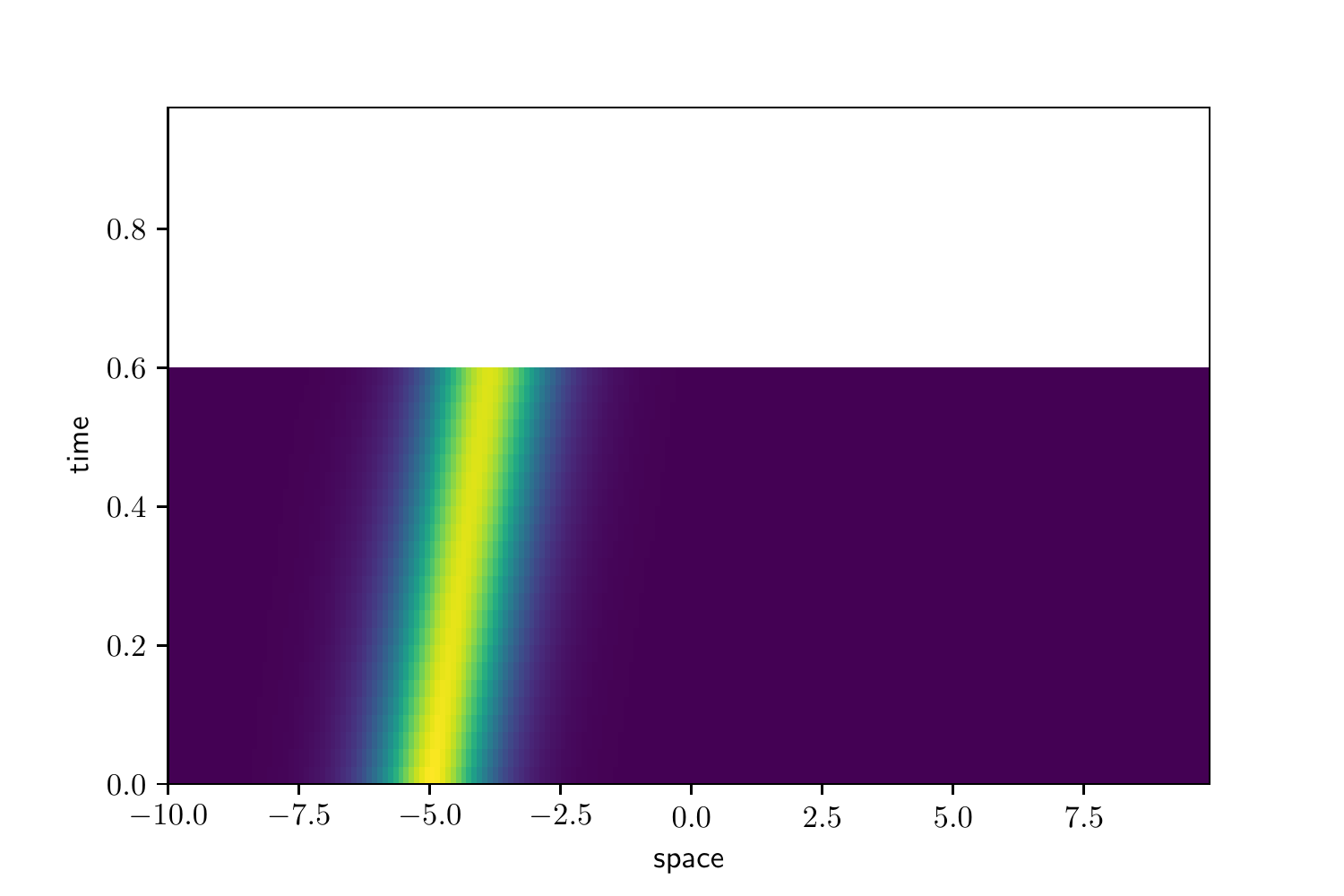}(b)\\
\includegraphics[width=.66\linewidth]{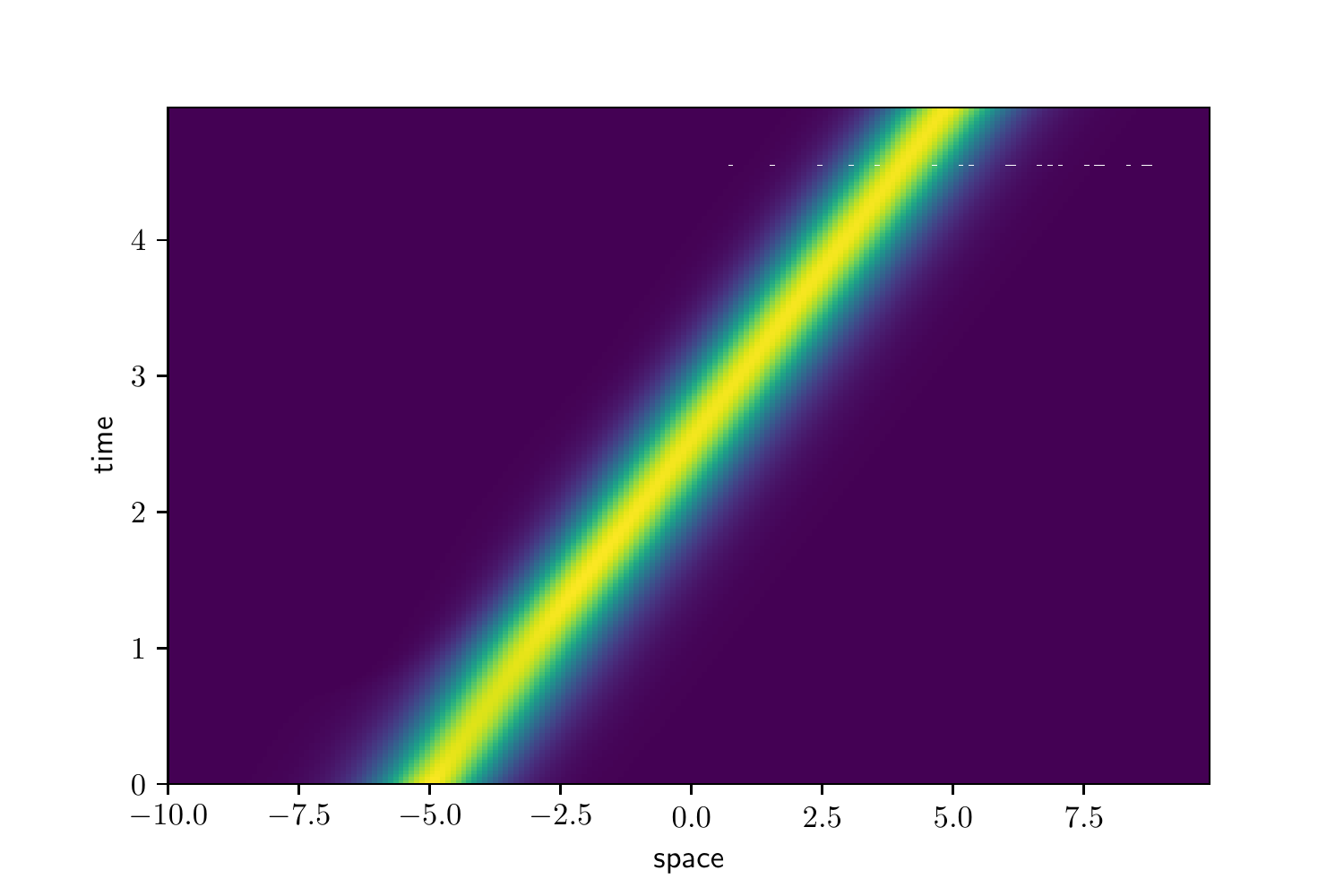}(c)
\caption{Positivity and stability issues in the numerical simulations of \eqref{eq:KS}. (a) Starting from a generic initial data, the numerical scheme quickly breaks down because the signal becomes negative at some point. The initial condition is shown in dashed line, and the final state in plain line (last time before numerical breakdown). (b) Aligning the initial data on the exact density and signal profiles $(\rho(z), S(z))$ \eqref{eq:KS-rho}--\eqref{eq:KS-S}, yields the same conclusion. The cell density is shown in space/time. The numerical breakdown occurs at approximately $t = 0.6$. (c) The propagation of the wave can be rescued by setting manually $S_{n+1} = \max(S_n,1E-12)$ after each time step, as in  \cite{hilpert_lattice-boltzmann_2005}. For all the figures, the parameters are $(d=1,\chi=2,D=0,k=1)$.\label{fig:KS}}
\end{center}
\end{figure}

Despite its elegance, the previous construction suffers from two drawbacks. First of all, the positivity of the  signal concentration $S$ is not guaranteed in the Cauchy problem. Actually, numerical solutions soon break down because of this positivity issue. This occurs starting from a generic initial data (Figure \ref{fig:KS}a), and even from the traveling wave solution $(\rho(z), S(z))$ given by the expressions \eqref{eq:KS-rho}--\eqref{eq:KS-S}, after accumulation of numerical errors (Figure \ref{fig:KS}b). Nevertheless, the positivity can be manually rescued by setting $S_{n+1} = \max(S_n,\epsilon)$ for some arbitrary threshold $\epsilon\ll 1$, as suggested in \cite{hilpert_lattice-boltzmann_2005}. In that case, the wave seems to propagate in a stable way in the long term, see Figure \ref{fig:KS}c.  

Second, and somewhat related, is the problem of stability of the wave constructed in Section \ref{sec:KS-construct}. Linear stability was addressed first in \cite{rosen_stability_1975}, where it was proven that the spectral problem admits  no {\em real} positive eigenvalue. However, the linearized problem is not self-adjoint, so that this preliminary result is largely incomplete from the perspective of stability. Few years later, it was proven in \cite{nagai_traveling_1991} that the (essential) spectrum of the linear operator  intersects the right half-plane, meaning that the wave is linearly unstable. The authors proved a refined instability result, when perturbations are restricted to a class of exponentially decreasing functions. Noticeably, their results cover both $D=0$ and $D>0$. This analysis has been largely extended in \cite{davis_absolute_2017,davis_spectral_2019}
where it was proven that the wave is either transiently (convectively) unstable, that is, the spectrum is shifted   in the open left half plane in a two-sided exponentially weighted function space \cite{sandstede_chapter_2002}, when $\chi>d$ is not too large, but it is absolutely unstable when $\chi$ is above some threshold, that is, $\frac\chi d>\beta^0_\mathrm{crit}(D)$, where, {\em e.g.} $\beta^0_\mathrm{crit}(0)$ is the unique real root above one of an explicit $10^\mathrm{th}$ order polynomial, see \cite[Theorem 2.1]{davis_spectral_2019}.

\begin{figure}
\begin{center}
\includegraphics[width=.66\linewidth]{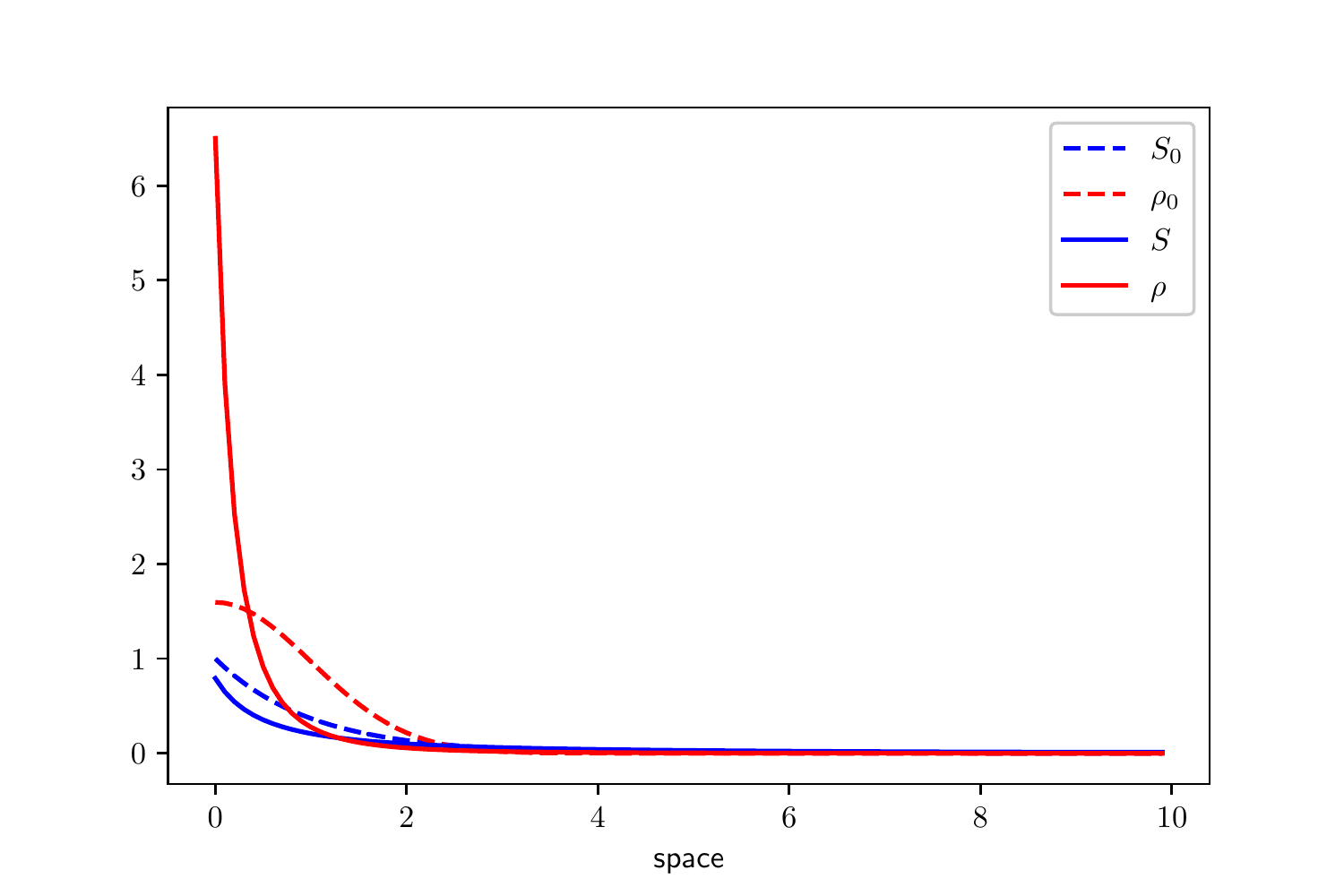}(a)\\
\includegraphics[width=.66\linewidth]{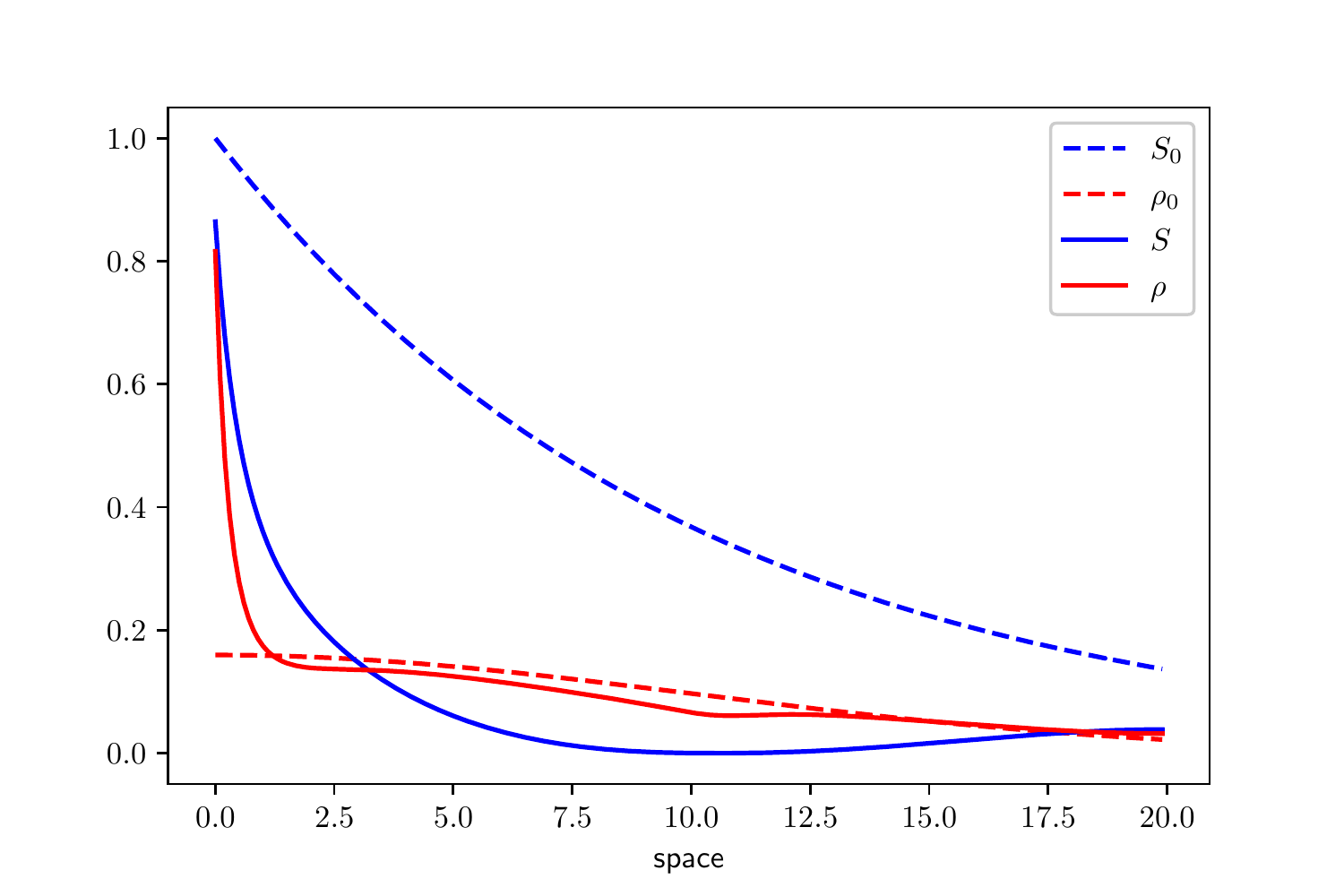}(b)
\caption{Numerical solutions of \eqref{eq:KS} with respectively Neumann boundary condition for $\rho$, and positive Dirichlet boundary condition for $S$  at the origin. (a) Local stability, as established in \cite{carrillo_boundary_2021} is illustrated numerically, for an initial condition chosen near the stationary state, and a relatively large diffusion of the chemical  ($d=1,\chi = 2,D=1,k=1$). (b) Nonetheless, the numerical solution may become nonpositive when the initial condition is far from the stationary state, and  diffusion of the chemical is not too large ($d=1,\chi = 2,D=0.25,k=1$). For each figure, the initial condition is shown in dashed line, and the final state in plain line (last time before numerical breakdown in (b)).  \label{fig:spike}}
\end{center}
\end{figure}

Recently, it has been established the existence and nonlinear  stability of stationary solutions for the problem \eqref{eq:KS} set on a half-line $\{x>0\}$, with respectively Neumann boundary condition for $\rho$, and positive Dirichlet boundary condition for $S$  at the origin  \cite{carrillo_boundary_2021}. The motivation comes from the study of spike solutions stabilized by a sustained amount of chemical concentration at the boundary.  The stability result in \cite{carrillo_boundary_2021} imposes quite stringent conditions on the decay of the initial data at $+\infty$. Nevertheless, local stability of the stationary spike does not preclude loss of positivity in the numerics when initiating the Cauchy problem with initial conditions far from equilibrium, see Figure \ref{fig:spike}.

\begin{remark} Many of the references mentioned above also discuss and analyze the case of a degenerate consumption rate  $\frac{\partial S}{\partial t} = D\frac{\partial^2 S}{\partial x^2} - k \rho S^m$  ($m<1$), without changing much of the global picture.

The case $m=1$ differs significantly, however. It can be viewed directly on the case $D = 0$ that the logarithmic gradient of the putative wave in the moving frame, that is, $\frac{d \log S}{d z}$ cannot have a positive limit as $z\to -\infty$, simply because it satisfies the relationship $-c \frac{d \log S}{d z} = -k\rho$, the latter being integrable. Consequently, advection cannot balance diffusion  at $-\infty$, preventing the existence of a traveling wave.  The same conclusion holds in the case $D>0$, for which $u = \frac{d \log S}{d z}$ is a homoclinic orbit of the following first-order equation
\[
\frac{du}{dz} = -\frac{c}{D} u - u^2  + \frac kD \rho
\,,
\]
that leaves the origin $u = 0$ at $z= -\infty$, and gets back to the origin
$u = 0$ at $z=+\infty$, see \cite[Proposition 6.3]{calvez_chemotactic_2019}.
\label{rem:1}
\end{remark}

\subsection{Variations on the Keller-Segel model}
\label{sec:variations}

As mentioned above, the seminal work \cite{keller_traveling_1971} gave rise to a wealth of modeling and analysis of traveling bands of bacteria. Many extensions were proposed soon after Keller and Segel's original paper, with various sensitivity functions (other than the logarithmic sensitivity), and various consumption rates. The models have the following general form,
\begin{equation}
\begin{cases}
\dfrac{\partial \rho}{\partial t} +  \dfrac{\partial }{\partial x}\left ( - d  \dfrac{\partial \rho}{\partial x} + \rho \bchi\left (S,\dfrac{\partial S}{\partial x}\right )    \right )  =  0\,,
\medskip\\
\dfrac{\partial S}{\partial t} = D \dfrac{\partial^2 S}{\partial x^2} - \bk(S,\rho)\,.
\end{cases}\label{eq:extended KS}
\end{equation}
where the chemotactic sensitivity $\bchi$ can be a function of both the signal concentration  and its gradient (as well as the diffusion coefficient $d$ -- dependency not reported here for the sake of clarity).
These variations were nicely reviewed by Tindall et al \cite{tindall_overview_2008}, and we are not going to comment them, but the contribution of  Rivero et al \cite{rivero_transport_1989}. The latter follows the approach of Stroock  \cite{stroock_stochastic_1974}, and Alt \cite{alt_biased_1980}. These approaches make the connection between  the individual response of bacteria to space-time environmental heterogeneities, and the macroscopic flux, hence making sense of the aforementioned averaging, by means of individual biases in the trajectories (see {\em e.g.} \cite{othmer_models_1988,
othmer_diffusion_2002,chalub_kinetic_2004,
filbet_derivation_2005,
chalub_model_2006,
bellomo_toward_2015}, and more specifically \cite{erban_individual_2004,
saragosti_mathematical_2010,si_pathway-based_2014,xue_macroscopic_2015} for bacterial populations). Interestingly, Rivero et al postulate a chemotactic advection speed $\bchi$ which is non-linear with respect to the chemical gradient at the macroscopic scale, namely
\begin{equation}\label{eq:rivero}
\bchi\left (S,\dfrac{\partial S}{\partial x}\right )  = \chi \tanh \left ( f(S) \dfrac{\partial S}{\partial x} \right ) \,,
\end{equation}  
where $f$ is a decreasing function containing the details of signal integration by a single cell. 

Up to our knowledge, none of the models in the long list of existing variations could exhibit traveling waves while preserving positivity of $S$ and keeping the total mass $\int_\mathbb{R} \rho$ constant (that is, ignoring growth). The minimal requirement for ensuring positivity would essentially be that the uptake function $\bk(S,\rho)$ is dominated by $S$ at small concentration, typically: $\limsup_{S\to 0} \frac{\bk(S,\rho)}{S}<\infty$. However, this intuitively leads to a shallow (logarithmic) gradient at the back of the wave, unable to guarantee the effective migration of cells left behind, see Remark \ref{rem:1}. Cell leakage has long been identified in the biological literature, but not considered as a major issue, see for instance a discussion  in \cite{fu_spatial_2018}, and also the addition of a linear growth term in \cite{saragosti_directional_2011} so that the loss of cells at the back is qualitatively compensated by cell division (for a realistic value of the division rate).

It is interesting to discuss the natural choice $-\bk(S,\rho) = -k S \rho$ (combined with logarithmic sensivity), which has been widely studied using tools from hyperbolic equations (after performing the Hopf-Cole transformation) by Z.A. Wang and co-authors, see the review \cite{wang_mathematics_2012}, and further stability results in \cite{jin_asymptotic_2013,li_stability_2014}.
The issue of shallow gradients is overcome by the boundary conditions at infinity, $\rho$ being uniformly positive at least on one side. Clearly, the traveling wave solutions are not integrable. This hints to the conflict of conservation of mass and chemical positivity which seem not concilable. 

This leakage effect is a major mathematical issue, because most of the analytical studies build upon the existence of a wave speed and a wave profile which is stationary in the moving frame.

\subsection{Beyond the Keller-Segel model: two scenarios for SGG}
\label{sec:beyond}

In the next two sections, we discuss two relevant modeling extensions, motivated by biological experiments, for which traveling waves exist and are expected to be stable. In the first scenario, cell leakage is circumvented by enhanced advection at the back of the wave, with an asymptotic constant value of the transport speed at $-\infty$. In the second scenario, cell leakage occurs, but it is naturally compensated by growth at the edge of the propagating front. 

For each scenario, we discuss briefly the biological motivations. Then we present the explicit construction of the traveling wave solutions, together with the formula for the wave speed. When possible, we discuss the connections with some other works in the literature.

\section{Scenario 1: strongest advection at the back}
\label{sec:scenario 1}

\begin{figure}
\includegraphics[width=\linewidth]{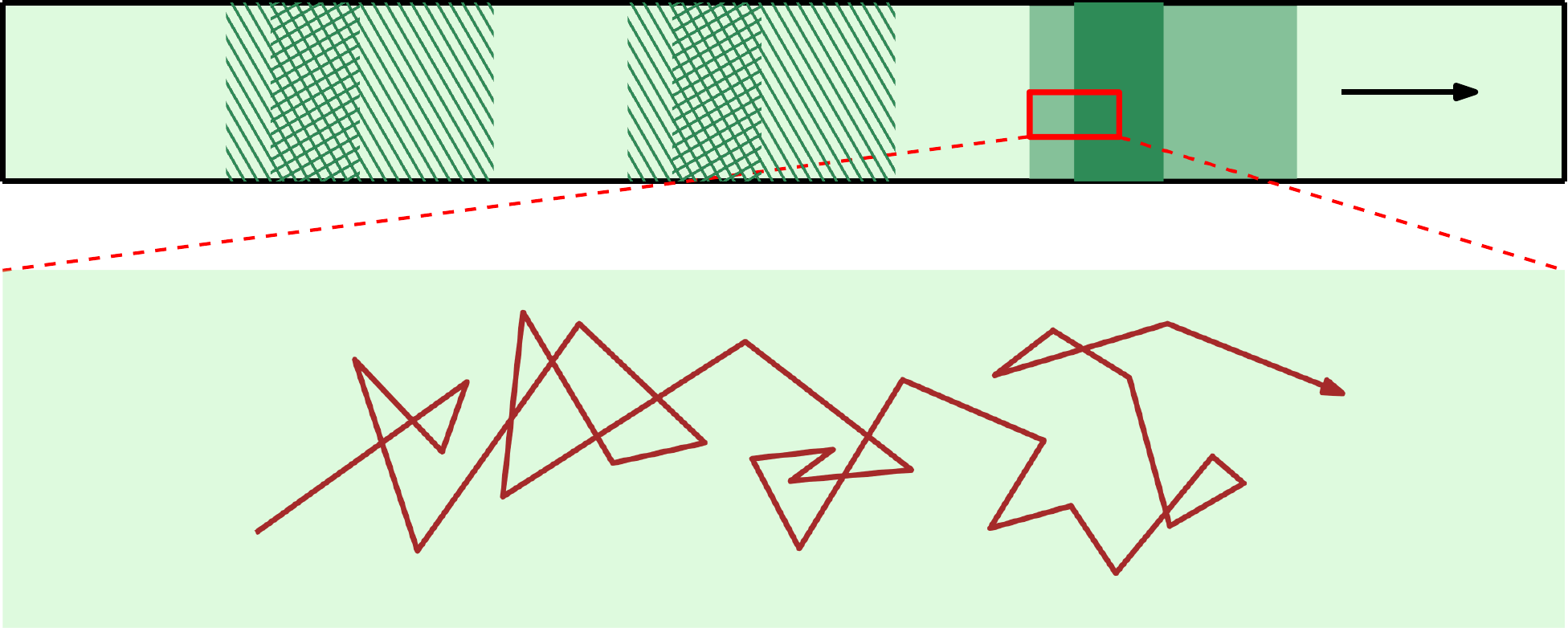}
\caption{Cartoon of the experiments performed in \cite{saragosti_mathematical_2010} and \cite{saragosti_directional_2011}. A band of bacteria is traveling from left to right in a microfluidic channel. Videomicroscopy allows tracking individual trajectories inside the wave, revealing heterogeneous behaviors: biases are stronger at the back of the wave than at the edge. 
\label{fig:band}}
\end{figure}

\begin{figure}
\begin{center}
\includegraphics[width = 0.8\linewidth]{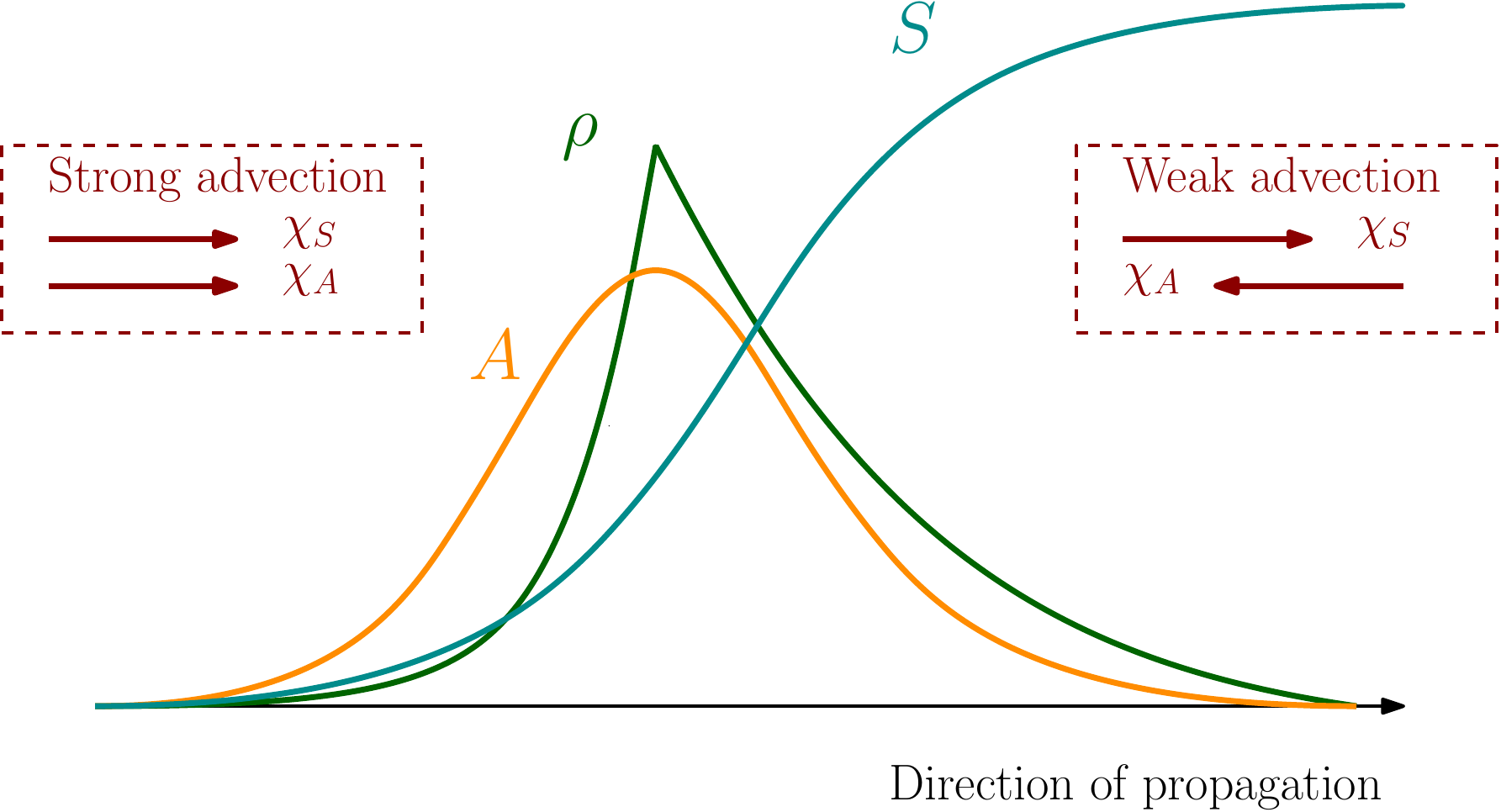}
\caption{Sketch of the chemical environment viewed by the cell density in model \eqref{eq:saragosti}. It is characterized by stronger advection at the back (the two signals have the same orientation), than at the edge (the two gradients have opposite orientations). When chemotactic speeds coincide ($\chi_S=\chi_A$), then we simply have diffusion on the right side of the peak.}\label{fig:gradients}
\end{center}
\end{figure}

\begin{figure}
\begin{center}
\includegraphics[width = 0.8\linewidth]{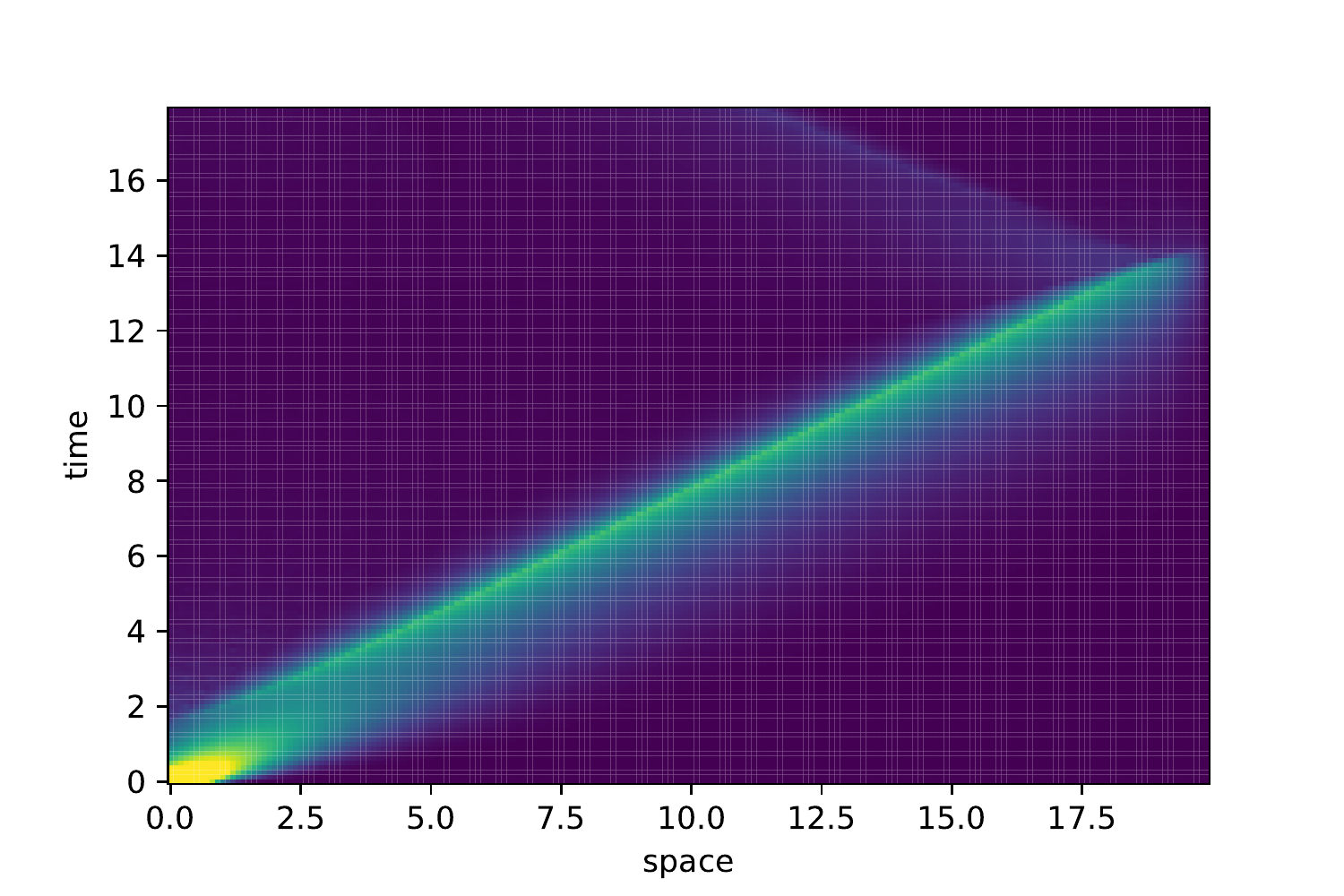}
\caption{Numerical simulation of model \eqref{eq:saragosti} for a half-gaussian initial density of bacteria.}\label{fig:saragosti}
\end{center}
\end{figure}

In this section, we present some  study performed a decade ago, revisiting original Adler's experiment, see Figure \ref{fig:band}. Inspired by massive tracking analysis, Saragosti et al \cite{saragosti_mathematical_2010} proposed a simple model for the propagation of chemotactic waves of bacteria, including two signals (see also \cite{xue_travelling_2011} for an analogous approach developed independently at the same time).  The macroscopic model is the following:
\begin{equation}
\begin{cases}
\dfrac{\partial \rho}{\partial t} +  \dfrac{\partial }{\partial x}\left ( - d  \dfrac{\partial \rho}{\partial x} + \rho \left ( \chi_S \sign \left ( \dfrac{\partial S}{\partial x} \right ) + \chi_A \sign \left ( \dfrac{\partial A}{\partial x} \right ) \right ) \right )  =  0\,,
\medskip\\
\dfrac{\partial S}{\partial t} = D_S \dfrac{\partial^2 S}{\partial x^2} - \bk(S,\rho) \,,
\medskip\\
\dfrac{\partial A}{\partial t} = D_A \dfrac{\partial^2 A}{\partial x^2} + \beta \rho - \alpha A\,.
\end{cases}\label{eq:saragosti}
\end{equation}
As compared to \eqref{eq:extended KS}, it is supplemented with a second chemical signal, $A$, which plays the role of a communication signal released by the cell population (hence, the source term $+\beta\rho$), and naturally degraded at a constant rate $\alpha>0$. Indeed, bacteria are known to secrete amino-acids, which play the role of a chemo-attractant as part of a positive feedback loop \cite{berg_e._2004,mittal_motility_2003}

Moreover,  bacteria are assumed to respond to the signal in a binary way at the macroscopic scale: the advection speed associated with each signal ($S,A$) can  take only two values, respectively  $\pm \chi_S$, $\pm \chi_A$, depending on the direction of the gradients. Then, the total advection speed is simply the sum of the two components. This was derived in \cite{saragosti_mathematical_2010} from a kinetic model at the mesoscopic scale, assuming a strong amplification during signal integration, see also \cite{calvez_chemotactic_2019} for a discussion. This can be viewed as an extremal choice of the advection speed proposed by Rivero et al \cite{rivero_transport_1989}, in the regime $f \to +\infty$ \eqref{eq:rivero}. The biophysical knowledge about the details of signal integration in bacteria {\em E. coli} have increased in the meantime \cite{tu_modeling_2008,
kalinin_logarithmic_2009,
jiang_quantitative_2010,
si_pathway-based_2012}. Actually, the logarithmic sensing is a good approximation in a fairly large range of signal concentrations. However, we retain this simple, binary, choice for theoretical purposes. 

As for the Keller-Segel model, traveling waves for \eqref{eq:saragosti} have the great advantage of being analytically solvable, essentially because the problem reduces to an equation with  piecewise constant coefficients. Introduce again the variable $z = x-ct$ in the moving frame at (unknown) speed $c$. Then, we have the following result:
\begin{theorem}[Saragosti et al \cite{saragosti_mathematical_2010}]
There exist a speed $c>0$, and a positive limit value $S_{-}<S_\init$, such that the system \eqref{eq:saragosti} admits a stationary  solution in the moving frame $(\rho(x - ct), S(x - ct), A(x-ct))$, such that $\rho$ is positive and integrable, $\int_\mathbb{R} \rho(z)\, dz = M$, $A$ decays to zero on both sides, and $S$ is increasing between the following limiting values 
\begin{equation*}
\begin{cases}
\lim_{z\to-\infty}S(z) = S_-\,, \quad\\
\lim_{z\to+\infty}S(z) = S_\init\,.
\end{cases}
\end{equation*}
Moreover, the speed $c>0$ is  determined by the following implicit relation,
\begin{equation}
\label{eq:speed}
\chi_S - c = \chi_A\dfrac{c}{\sqrt{c^2 + 4 \alpha D_A}}\, .
\end{equation}
\end{theorem}

\begin{proof}
Contrary to the proof of Theorem \ref{th:KS}, the wave speed $c$ cannot be computed by a direct argument.

As a preliminary step, we should prescribe the environmental conditions, as they are expected heuristically to be seen by the bacteria, see Figure \ref{fig:gradients}.
On the one hand, we seek an increasing profile $S$, hence  $\sign\left ( \frac{dS}{dz} \right ) = +1$, and the equation on the density profile  $\rho$ is decoupled from the dynamics of $S$.  On the other hand, we assume that 
the communication signal $A$ reaches a unique maximum, that can be set at $z = 0$ by translational invariance. The validation of this ansatz, {\em a posteriori}, will set the equation for $c$ \eqref{eq:speed}. 

The equation for $\rho$ has now piecewise constant coefficients in the moving frame,
\begin{equation*}
- c \dfrac{d \rho}{dz} +  \dfrac{d }{dz}\left ( - d  \dfrac{d \rho}{dz} + \rho \left ( \chi_S  + \chi_A \sign \left ( -z \right ) \right ) \right )  =  0\,.
\end{equation*}
%
Hence, $\rho$ is a combination of two exponential functions, 
\begin{equation*}
\rho(z) = 
\begin{cases}
\exp\left(\lambda_- z \right) & \text{for $z <0$,\quad $\lambda_- = \dfrac{ - c + \chi_S + \chi_A}{d}$ \; (signals are aligned),}\medskip\\
\exp\left(-\lambda_+ z \right) & \text{for $z >0$, \quad $\lambda_+ = \dfrac{  c - \chi_S + \chi_A }{d}$ \;   (signals are competing).}
\end{cases}
\end{equation*}
Next, the attractant concentration $A$ can be computed explicitly, by convolving the source term $\beta \rho$ with the fundamental solution of the elliptic operator $-c \frac{d}{dz} - D_A \frac{d^2}{dz^2} + \alpha$, denoted by $\mathcal A$, that is, $A = \beta \mathcal A*\rho$. Coincidentally, $\mathcal A$ shares the same structure as $\rho$, namely $\mathcal A(z) = a_0 \exp(\mu_- z)$ for $z<0$ and $\mathcal A(z) = a_0\exp(-\mu_+ z)$ for $z>0$, with $\mu_\pm = \frac1{2D_A}\left ( \pm c + \sqrt{c^2 + 4 \alpha D_A} \right )$, and $a_0$ is a normalizing factor. 

It remains to check the preliminary ansatz, that is, $A$ changes monotonicity at $z=0$. A straightforward computation yields
\begin{equation*}
\frac{dA}{dz}(0) = \beta a_0 
\left( - \dfrac{1 }{1 + \lambda_-/\mu_+ } + \dfrac{1}{ 1 + \lambda_+/\mu_-} \right)\, .
\end{equation*}
Therefore, the construction is complete, provided $\lambda_-\mu_- = \lambda_+\mu_+$, 
which is equivalent to \eqref{eq:speed}. 
\end{proof}

To partially conclude, let us highlight the fact that  cohesion in the wave  is guaranteed by the local aggregation signal $A$. To put things the other way around, in the absence of the driving signal $S$, the cells can aggregate thanks to the secretion of $A$, and the density reaches a stationary state (standing wave). In turn, this cohesive state can travel (with some deformation) in the presence of the (self-generated) driving signal $S$. To make the link with SGG in developmental biology \cite{dona_directional_2013}, let us point to the modeling study  \cite{streichan_collective_2011} which is devoted to the  migration of cell collectives in the lateral line during development of the zebrafish. There, it is assumed that the rod of cells maintains its shape {\em per se} with a constant length which is a parameter of the model, see also \cite{carmona-fontaine_complement_2011} for biological evidence of cell attraction during collective motion.

\section{Scenario 2: cell leakage compensated by growth}

\label{sec:scenario 2}

\begin{figure}
\begin{center}
\includegraphics[width=.8\linewidth]{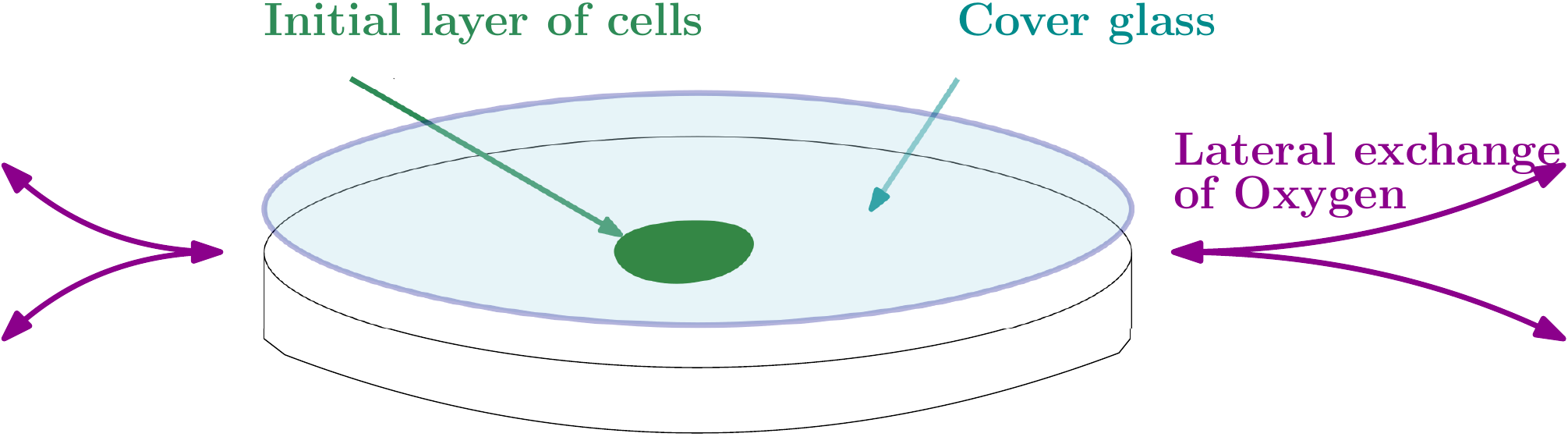}(a)\vspace{25pt}\\
\includegraphics[width=.8\linewidth]{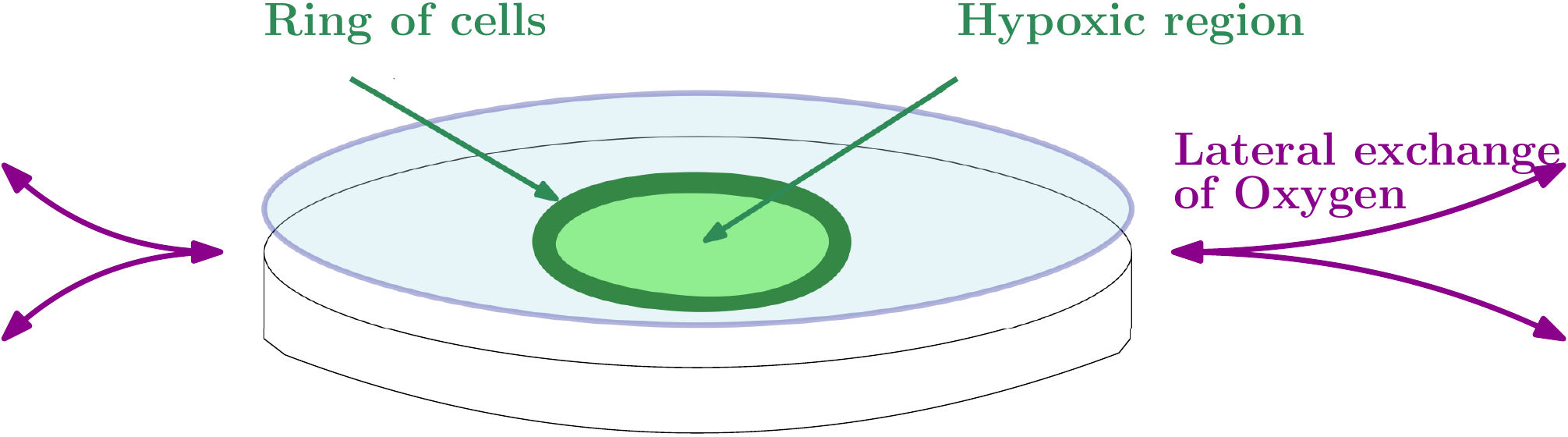}(b)
\caption{Schematic view of the experimental set-up in \cite{cochet-escartin_hypoxia_2021}. (a) An initial layer of Dicty cells is deposited at the center of the plate, and covered with a large glass coverslip (after \cite{deygas_redox_2018}). This vertical confining reduces drastically the inflow of oxygen within the plate, by restricting it to lateral exchanges. (b) Soon after the beginning of the experiment, a ring of cells emerges, which is traveling over several days at constant speed with a well-preserved shape. The moving ring consumes almost all the available oxygen, so that the center of the colony is at very low concentration, below 1\%.   
\label{fig:coverglass}}
\end{center}
\end{figure}

\begin{figure}
\begin{center}
\includegraphics[width=.8\linewidth]{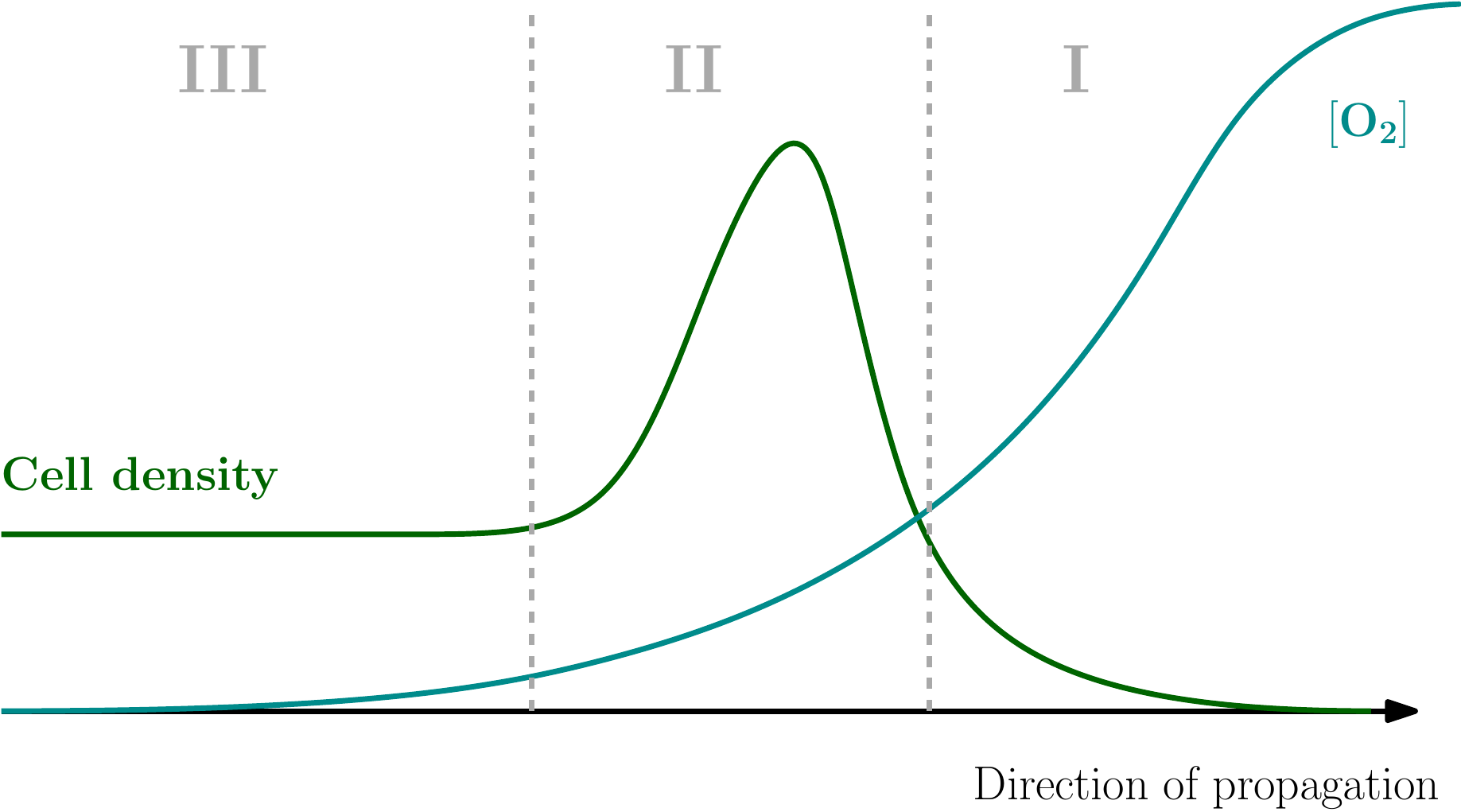}(a)\bigskip\\
\includegraphics[width=.8\linewidth]{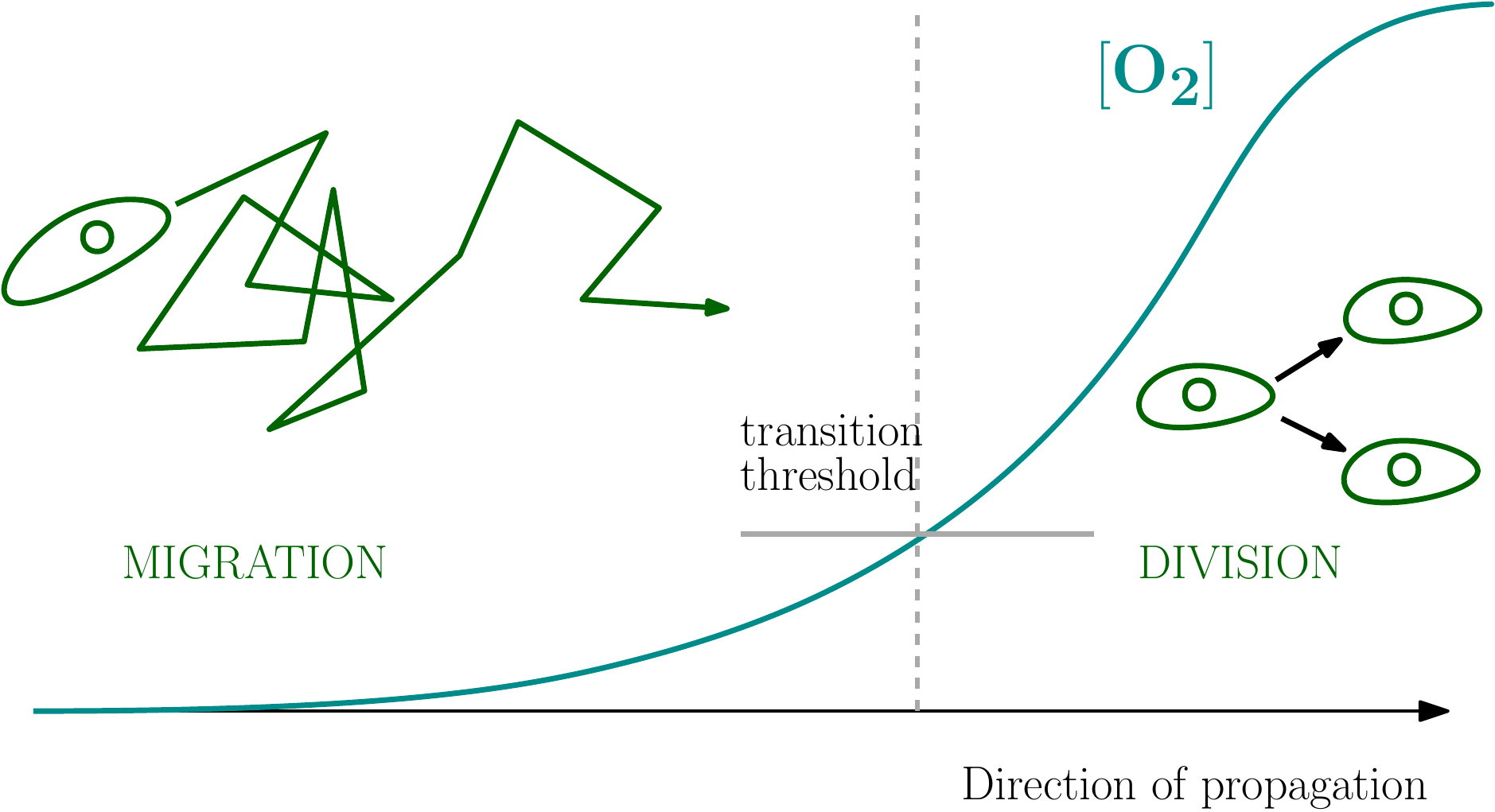}(b)
\caption{Graphical description of the 'go-or-grow' model \eqref{eq:extended KS growth}. (a) Individual cell tracking in \cite{cochet-escartin_hypoxia_2021} shows different cell behaviours depending on the relative position with respect to the tip of the ring: (I) ahead of the moving ring, cells exhibit unbiased motion, together with division events; (II) inside the ring, cells exhibit clear directional motion (which indeed results in the formation and maintenance of the ring); (III) the trail of cells which are left behind exhibit unbiased motion, again, with more persistent trajectories (but this last observation is neglected in the model, because it was shown to have limited effect). (b) We hypothesize a single transition threshold $S_0$ such that cells can divide above the threshold, while they move preferentially up the gradient below the threshold, when oxygen is limited. The unbiased component of cell motion (diffusion) is common to both sides of the threshold.    
\label{fig:minimal}}
\end{center}
\end{figure}

In this section, we present a recent  model of SGG, including localized (signal-dependent) growth \cite{cochet-escartin_hypoxia_2021}. This work was motivated by aeroactic waves of Dicty observed in vertically confined assays, in which oxygen is consumed by the cells and is soon limited at the center of the colony, see Figure \ref{fig:coverglass}. We refer to \cite{cochet-escartin_hypoxia_2021} for the experimental details. The model introduced in \cite{cochet-escartin_hypoxia_2021} was referred to as a "go-or-grow" model, a term coined in a previous work by Hatzikirou et al \cite{hatzikirou_go_2012} in the context of modeling cell  invasion in brain tumors. There, the basic hypothesis was that cells could switch between two states, or phenotypes: a migrating state 'go' (with enhanced random diffusion), and a proliferating state 'grow' (with enhanced rate of division), following previous works in the same context (see {\em e.g.} \cite{fedotov_migration_2007}). In \cite{hatzikirou_go_2012} it was assumed that hypoxia (lack of oxygen) triggers the switch in the long term dynamics of the system, by selection of the migrating phenotype, but in a global manner (oxygen supply was accounted for via the constant carrying capacity, as one parameter of the cellular automaton). Later contributions considered PDE models with density-dependent switch (see \cite{stepien_traveling_2018}, as opposed to \cite{fedotov_migration_2007} where the switching rate is not modulated, and also the experimental design of density-dependent motility in bacteria \cite{liu_sequential_2011}). 

In \cite{cochet-escartin_hypoxia_2021}, the go-or-grow hypothesis was revisited, by studying an expanding ring of Dicty cells, with limited supply of oxygen. Figure \ref{fig:minimal}a shows the cell density profile, as it is observed in experiments. Figure \ref{fig:minimal}b summarizes the minimal assumption of an oxygen-dependent switch, as proposed in \cite{cochet-escartin_hypoxia_2021}. It was hypothesized that the transition between the proliferating state and the migrating state is modulated by the level of oxygen, with a sudden change of phenotype at some threshold $S_0$. Above this threshold, when oxygen is available in sufficient quantity, cells exhibit slow random (diffusive) motion and divide at some constant rate. Below this threshold, when oxygen is limited, cells stop dividing and move preferentially up the oxygen gradient. The latter hypothesis (directional motion) is different from the aforementioned go-or-grow models \cite{fedotov_migration_2007,
hatzikirou_go_2012,stepien_traveling_2018}. It is consistent with the observations of individual tracking within the cell population in the bulk of the wave in \cite{cochet-escartin_hypoxia_2021}. 

The following model recapitulates these assumptions,
\begin{equation}
\begin{cases}
\dfrac{\partial \rho}{\partial t} +  \dfrac{\partial }{\partial x}\left ( - d  \dfrac{\partial \rho}{\partial x} + \rho \bchi\left (S,\dfrac{\partial S}{\partial x}\right )    \right )  =  \br(S) \rho\,,
\medskip\\
\dfrac{\partial S}{\partial t} = D \dfrac{\partial^2 S}{\partial x^2} - \bk(S,\rho) \,.
\end{cases}\label{eq:extended KS growth}
\end{equation}
with the specific choice 
\begin{equation}\label{eq:gogrow}
\bchi\left (S,\dfrac{\partial S}{\partial x}\right ) = \chi \sign\left( \dfrac{\partial S}{\partial x}\right ) \mathbf{1}_{S<S_0}\,, \quad \br(S) = r \mathbf{1}_{S>S_0}\,.
\end{equation}
This can be viewed as another variation of \eqref{eq:extended KS} including growth. It can also be viewed as an extension of the celebrated F/KPP equation, with a signal-dependent growth saturation, and including  advection (we refer to  \cite{ryzhik_traveling_2008,
calvez_traveling_2018,zeng_logarithmic_2019} and references therein for more classical synthesis of the F/KPP equation and the Keller-Segel model of cellular aggregation). Interestingly, an analogous model was proposed in \cite{franz_travelling_2013}, following a general motivation, and beginning with the statement that proliferation is necessary to sustain wave propagation. 
As compared with \eqref{eq:extended KS growth}--\eqref{eq:gogrow},  in the latter work, the reproduction rate $\br$  is signal-dependent with a linear dependency, and there is no threshold on the chemosensitivity $\bchi$ which is simply a linear function of the gradient $\frac{\partial S}{\partial x}$.  As a consequence, the wave speed cannot be calculated analytically, in constrast with \eqref{eq:extended KS growth}--\eqref{eq:gogrow} (see Theorem \ref{th:cochet} below). 

Before we show the construction of traveling wave solutions for \eqref{eq:extended KS growth}--\eqref{eq:gogrow}, let us comment on the reason why such solutions can exist. The expected density profile exhibits a  plateau of cells left behind the wave, see  Figure \ref{fig:minimal}a. In the vertical confining assay  experiment with Dicty, this corresponds to cells that are still highly motile, but have lost the propension to move directionally. They cannot keep pace with the self-generated oxygen gradient. The increasing amount of cells which are left behind is compensated by the growth at the edge of the pulse. This localized growth term (above the oxygen threshold) creates a flux term (negative flux in the moving coordinate) which is  key to the mathematical construction of the wave. 

We can be more precise about the negative flux issued from cell division by looking at the traveling wave equation \eqref{eq:extended KS growth}--\eqref{eq:gogrow} in the moving coordinate $z = x-ct$. 
\begin{equation}\label{eq:gogrowTW}
- c\dfrac{d \rho}{dz} +  \dfrac{d}{dz}\left ( - d  \dfrac{d \rho}{dz} + \rho \bchi\left (S,\dfrac{d S}{dz}\right )    \right )  =  \br(S) \rho\,.
\end{equation}
Below the oxygen threshold, $S<S_0$, the right-hand-side vanishes, and we are left with a constant flux, 
\begin{equation}\label{eq:flux J}
- c \rho  - d  \dfrac{d \rho}{dz} + \rho \bchi\left (S,\dfrac{d S}{dz}\right ) = -J \,.
\end{equation}
By integrating \eqref{eq:gogrowTW} on $\{S>S_0\}$, and using the continuity of the flux at the interface $\{S=S_0\}$, we find
\begin{equation}\label{eq:flux neg}
J = r \int_{\{S>S_0\}} \rho(z)\, dz \,.
\end{equation}
Note that the continuity of the flux is a pre-requisite for the well-posedness of \eqref{eq:extended KS growth}--\eqref{eq:gogrow}, see \cite{demircigil_notitle_nodate} for a rigorous  analysis of this problem, and unexpected mathematical subtleties.

\begin{figure}
\begin{center}
\includegraphics[width = 0.8\linewidth]{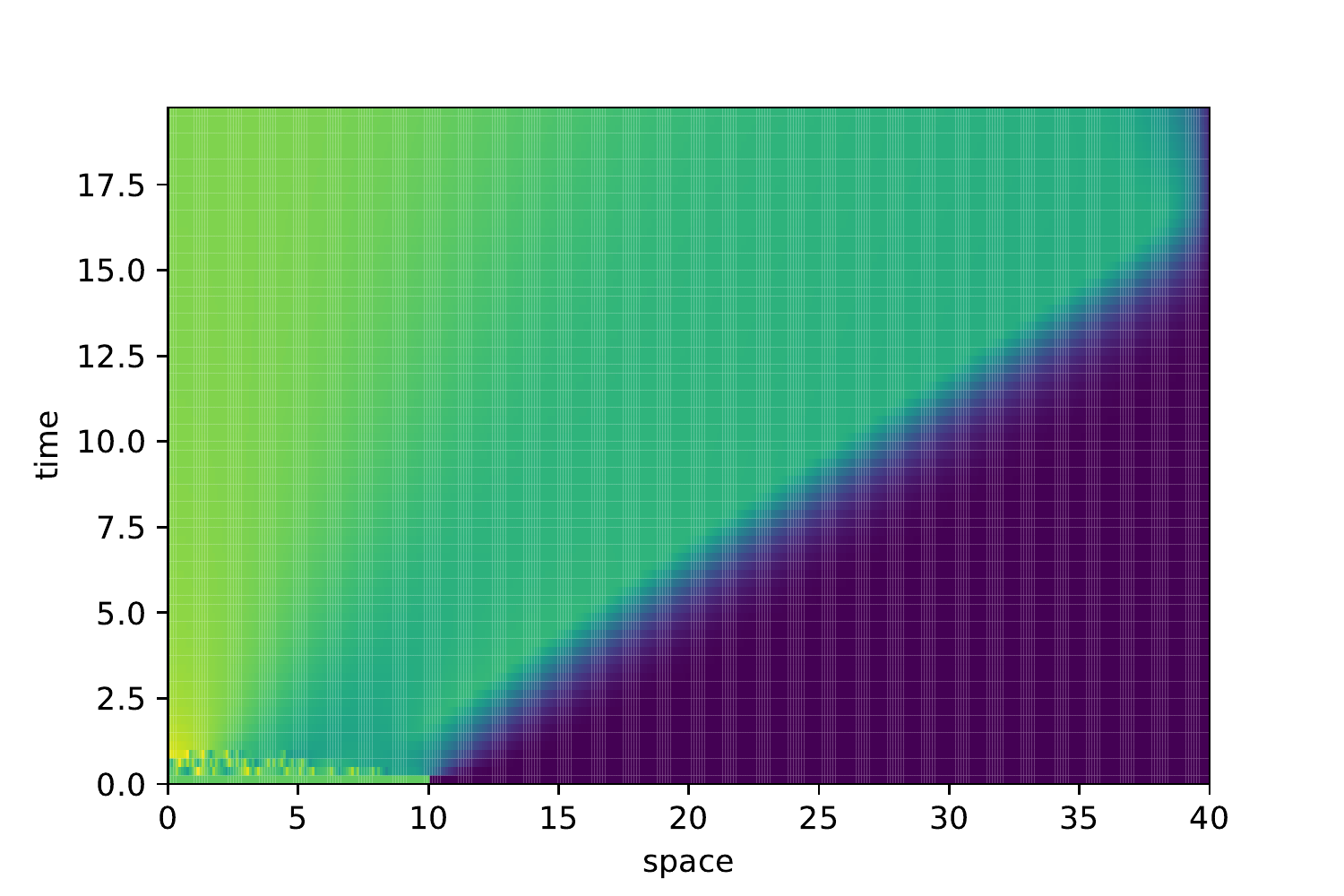}
\caption{Numerical simulation of model \eqref{eq:saragosti} for an initial plateau of cells restricted to the interval $\{x<10\}$.}\label{fig:cochet}
\end{center}
\end{figure}

%

\begin{theorem}[Cochet et al \cite{cochet-escartin_hypoxia_2021}, Demircigil \cite{demircigil_notitle_nodate}]
\label{th:cochet}
There exist a speed $c>0$, and a positive limit value $\rho_{-}>0$, such that the system  \eqref{eq:extended KS growth}--\eqref{eq:gogrow}, admits a stationary solution in the moving frame  $(\rho(x-ct), S(x-ct))$, such that $\rho$ and $S$ have the following limiting values
\begin{equation*}
\begin{cases}
\lim_{z\to-\infty}\rho(z) = \rho_-\,,  \\
\lim_{z\to+\infty}\rho(z) = 0\,,
\end{cases}
\quad 
\begin{cases}
\lim_{z\to-\infty}S(z) = 0\,,  \\
\lim_{z\to+\infty}S(z) = S_\init\,.
\end{cases}
\end{equation*} 
Moreover, the speed is given by the following dichotomy
\begin{equation}
c = \begin{cases}
2\sqrt{r d} & \text{if $\chi\leq\sqrt{rd}$}\,,\\
\chi + \dfrac{ rd}{\chi} & \text{if $\chi\geq\sqrt{rd}$}\,.
\end{cases}\label{eq:cochet}
\end{equation}
\end{theorem}

Interestingly, the dichotomy in \eqref{eq:cochet} depends on the relative values of the advection speed (up the gradient) $\chi$, and half the reaction-diffusion speed of the F/KPP equation $\sqrt{r d}$. When the aerotactic biases are small (low advection speed $\chi$), then the wave is essentially driven by growth and diffusion. When biases are large, then the wave is mainly driven by aerotaxis. This has interesting implications in terms of maintenance of genetic diversity inside the wave (see \cite{roques_allee_2012,
birzu_fluctuations_2018} for diversity dynamics among reaction-diffusion traveling waves). In fact, the so-called dichotomy between {\em pulled} and {\em pushed} waves is at play here, see \cite{cochet-escartin_hypoxia_2021,demircigil_notitle_nodate} for more details and discussion. 

In contrast with the original Keller-Segel model \eqref{eq:speed KS}, the wave speed does not depend on the features of oxygen consumption and diffusion. 

\begin{proof}
As in Section \ref{sec:scenario 1}, the wave speed is not given {\em a priori}. We seek a monotonic oxygen profile, such that $\frac{dS}{dz}>0$. Therefore, the first equation reduces to
\begin{equation*}
-c \dfrac{d \rho}{d z} - d \dfrac{d^2 \rho}{d z^2}   
+ \dfrac{d }{d z} \left\{\begin{array}{ll} 0 & \text{if $S>S_0$} \smallskip \\
 \chi \rho & \text{if $S<S_0$} 
\end{array}\right\} \\
= \left\{\begin{array}{ll} r \rho & \text{if $S>S_0$} \smallskip \\
 0 & \text{if $S<S_0$} 
\end{array}\right\} \,.
\end{equation*}
By translational invariance, we assume that $S = S_0$ occurs at $z=0$. 

For $z<0$, we have by \eqref{eq:flux J}--\eqref{eq:flux neg},
\begin{equation}\label{eq:ODEz-}
  d \dfrac{d \rho}{d z} = J + (\chi - c)\rho\,,\quad J>0  \,.
\end{equation}
Suppose that $c\leq \chi$. Then, $d \frac{d \rho}{d z}\geq J >0$, which is a contradiction with the positivity of $\rho$. Hence, we must have $c>\chi$. The solution of \eqref{eq:ODEz-} is unbounded unless it is constant, that is $\rho = \frac{J}{c-\chi}$, and this is the natural choice we make for the construction. 


For $z>0$ we have the standard linear problem arising in the F/KPP equation (at small density),
\begin{equation*}
-c \dfrac{d \rho}{d z} - d \dfrac{d^2 \rho}{d z^2} = r \rho\,.
\end{equation*}
We look for exponential solutions $\exp(-\lambda z)$. The characteristic equation, $d\lambda^2 - c \lambda + r= 0$ has real roots when $c^2 \geq 4rd$. Then, we proceed by dichotomy. 

$\diamond$ The case $ c= 2\sqrt{rd}$. The general solution for $z>0$ is of the form $(a + b z ) \exp(-\lambda z)  $, with $\lambda = \sqrt{\frac{r}{d}}$ the double root. The constant $a$ coincides with $\frac{J}{c-\chi}$ by continuity of the density (its value does not really matter here).  Continuity of the flux at the interface $z=0$ yields $-d(b-a\lambda ) = \chi a$, hence $ b d  = a(\sqrt{rd} - \chi)$. Thus, the solution is admissible ($b\geq 0$) if, and only if $\chi \leq \sqrt{rd}$.   

$\diamond$ The case $ c> 2\sqrt{rd}$. Standard arguments in the construction of reaction-diffusion traveling waves imply to select the sharpest decay on the right side \cite{aronson_multidimensional_1978,van_saarloos_front_2003}, namely $\rho = a\exp(-\lambda z)$, with  $\lambda = \frac1{2d} \left ( c + \sqrt{c^2 - 4 r d} \right )$. Continuity of the flux at the interface now writes $-d(-a\lambda) = \chi a$, which is equivalent to 
\begin{equation*}
2\chi - c =  \sqrt{c^2 - 4 r d} \quad \Leftrightarrow\quad \left (c = \chi + \dfrac{rd}\chi\right ) \; \& \; \left ( \chi > \frac c 2 \right )\,. 
\end{equation*}
It must be  checked {\em a posteriori} that $c> 2\sqrt{rd}$, which is immediate. The last inequality constraint ensures that $\chi>\sqrt{rd}$, in contrast with the other side of the dichotomy. 

Thus, the construction is complete. 
\end{proof}


The wavefront constructed above appears to be numerically stable, driving the long-time asymptotics, see Figure \ref{fig:cochet}. However, the very strong advection  at the back of the wave creates a decreasing density profile, which is actually constant at the back of the wavefront, in contrast with the experiments showing a non-monotonic pulse (Figure \ref{fig:minimal}). Several extensions were discussed in \cite{cochet-escartin_hypoxia_2021}.    

\subsubsection*{Logarithmic sensitivity.}
Below, we discuss a natural, yet original, extension of the previous result, restoring the logarithmic gradient in the advection term. More precisely, we consider \eqref{eq:extended KS growth} again, with the following choice of functions, instead of   \eqref{eq:gogrow}
\begin{equation}\label{eq:gogrowlog}
\bchi\left (S,\dfrac{\partial S}{\partial x}\right ) = \chi \log\left( \dfrac{\partial S}{\partial x}\right ) \mathbf{1}_{S<S_0}\,, \quad \br(S) = r \mathbf{1}_{S>S_0}\,.
\end{equation}
We present below a preliminary  result about the existence of traveling waves, followed by heuristic arguments about the determination of the speed, and some numerical investigation. 

\begin{theorem}\label{th:roxana}
Assume $D = 0$, and $\bk(S,\rho)  = k \rho S$ for some $k>0$.
There exists a speed $c>0$, and a positive limit value $\rho_{-}>0$, such that the system  \eqref{eq:extended KS growth}--\eqref{eq:gogrowlog}, admits a stationary solution in the moving frame  $(\rho(x-ct), S(x-ct))$, such that $\rho$ and $S$ have the following limiting values
\begin{equation*}
\begin{cases}
\lim_{z\to-\infty}\rho(z) = \rho_-\,, \\
\lim_{z\to+\infty}\rho(z) = 0\,,
\end{cases}
\quad 
\begin{cases}
\lim_{z\to-\infty}S(z) = 0\,,\\
\lim_{z\to+\infty}S(z) = S_\init\,.
\end{cases}
\end{equation*} 
Moreover, the speed is given by the following dichotomy
\begin{equation}
c = 2 \sqrt{r \max \left \{d, \chi \log\left ( \dfrac{S_\init}{S_0} \right )\right \}}\,.\label{eq:roxana}
\end{equation}
\end{theorem}

\begin{proof}
We proceed similarly as in the proof of the previous statement. The assumption $D = 0$ enables expressing the logarithmic gradient in terms of the density:
\begin{equation}\label{eq:log S}
-c \dfrac{d (\log S)}{dz} = - k \rho\,.
\end{equation}

For $z<0$ we have a constant (negative) flux at equilibrium in the moving frame \eqref{eq:flux J},
\begin{equation}\label{eq:J}
- c \rho - d   \dfrac{d \rho}{dz} + \chi \rho  \dfrac{d (\log S)}{dz} = - J < 0 \,.
\end{equation}
Combining \eqref{eq:log S} and \eqref{eq:J}, we get the ODE satisfied by the cell density profile at the back:
\begin{equation}\label{eq:ODE z-}
d   \dfrac{d \rho}{dz} = -  c \rho + \dfrac{k\chi}c \rho^2 +  J \,.
\end{equation}
This ODE comes with a sign condition, for the discriminant of the right-hand-side to be non-negative (otherwise $\rho$ cannot be positive for all $z<0$ when $\frac{d\rho}{dz}$ is uniformly positive), that is
\begin{equation}\label{eq:sign J}
 \dfrac{c^3}{4 k \chi} \geq J\,.
\end{equation}
This condition is complemented by the integration of \eqref{eq:log S} over $\{z>0\}$:
\begin{equation*}
c\log \left ( \dfrac{S_\init}{S_0} \right ) = k  \int_{0}^{+\infty} \rho(z)\, dz  = \frac{k}{r} J\,,
\end{equation*}
where the last identity follows from \eqref{eq:flux neg}. This yields the constraint 
\begin{equation}\label{eq:add restr}
\dfrac{c^3}{4 r \chi} \geq c\log \left ( \dfrac{S_\init}{S_0} \right ) \quad \Leftrightarrow\quad c^2 \geq 4r\chi\log \left ( \dfrac{S_\init}{S_0} \right )\,.
\end{equation}
This is one part of the condition in \eqref{eq:roxana}. The second part comes naturally from the constraint on the characteristic equation on  $\{z>0\}$, namely $c^2 \geq 4 r d$. It can be shown by simple phase plane analysis that admissible solutions exist in both cases when the inequality \eqref{eq:roxana} is an equality. 
\end{proof}

The previous analysis calls for a few comments:
\begin{enumerate}
\item Contrary to the former construction in Theorem \ref{th:cochet}, the latter construction does not come naturally with an equation for $c$. This is because there is no clear way to remove one degree of freedom on $\{z<0\}$ under the sign condition \eqref{eq:sign J}. Indeed, the solution of \eqref{eq:ODE z-} is naturally bounded for any intial condition, in opposition to \eqref{eq:ODEz-}.
\item Surprisingly, the additional restriction \eqref{eq:add restr} results from conditions imposed on the solution {\em at the back of the wave} on $\{z<0\}$, in opposition with the standard case, say for F/KPP and related equations, where it always come from conditions on $\{z>0\}$ (as it is the case for the classical restriction $c^2\geq 4 r d$).   
\end{enumerate}

\begin{figure}
\begin{center}
\includegraphics[width=.66\linewidth]{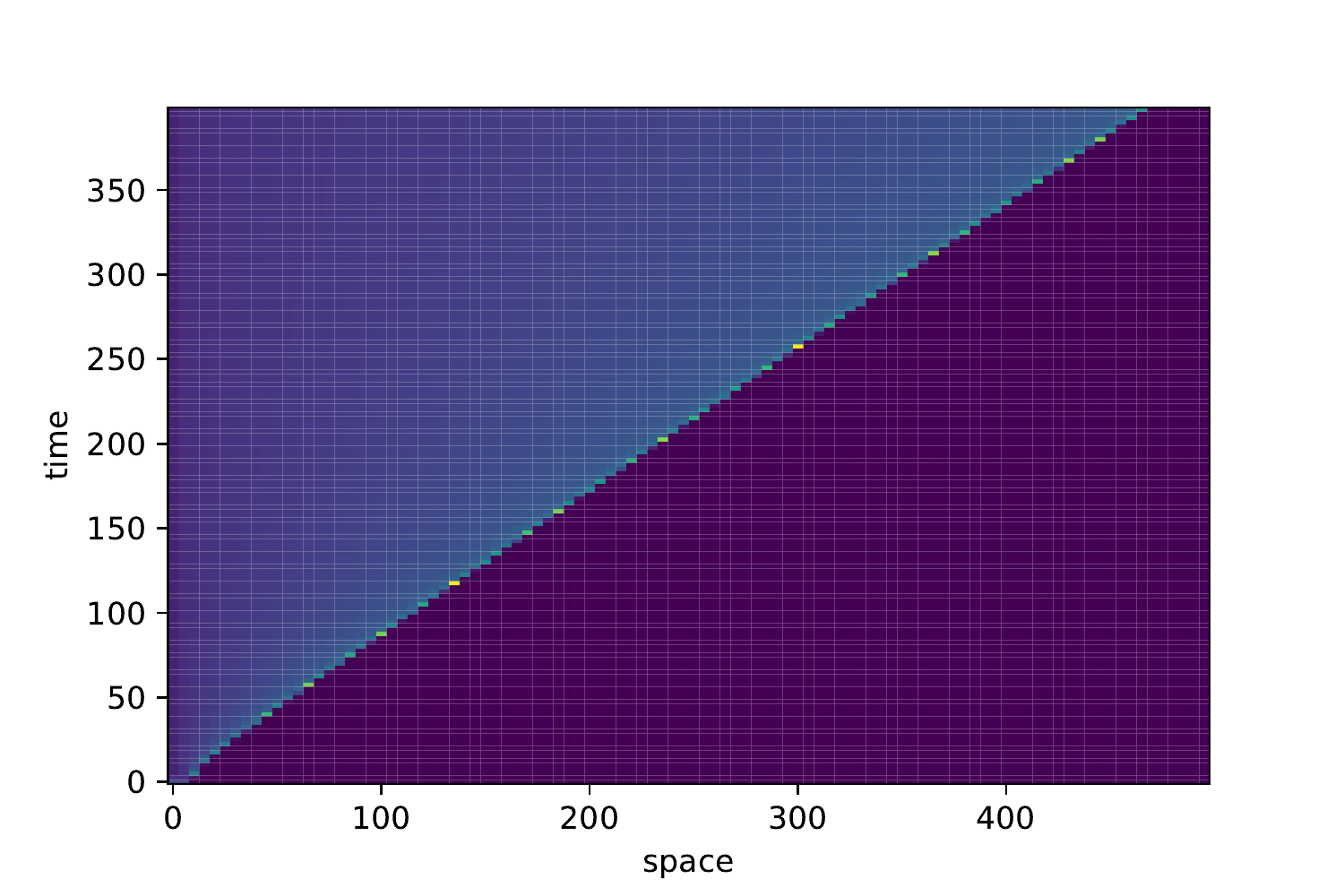}(a)\\
\includegraphics[width=.66\linewidth]{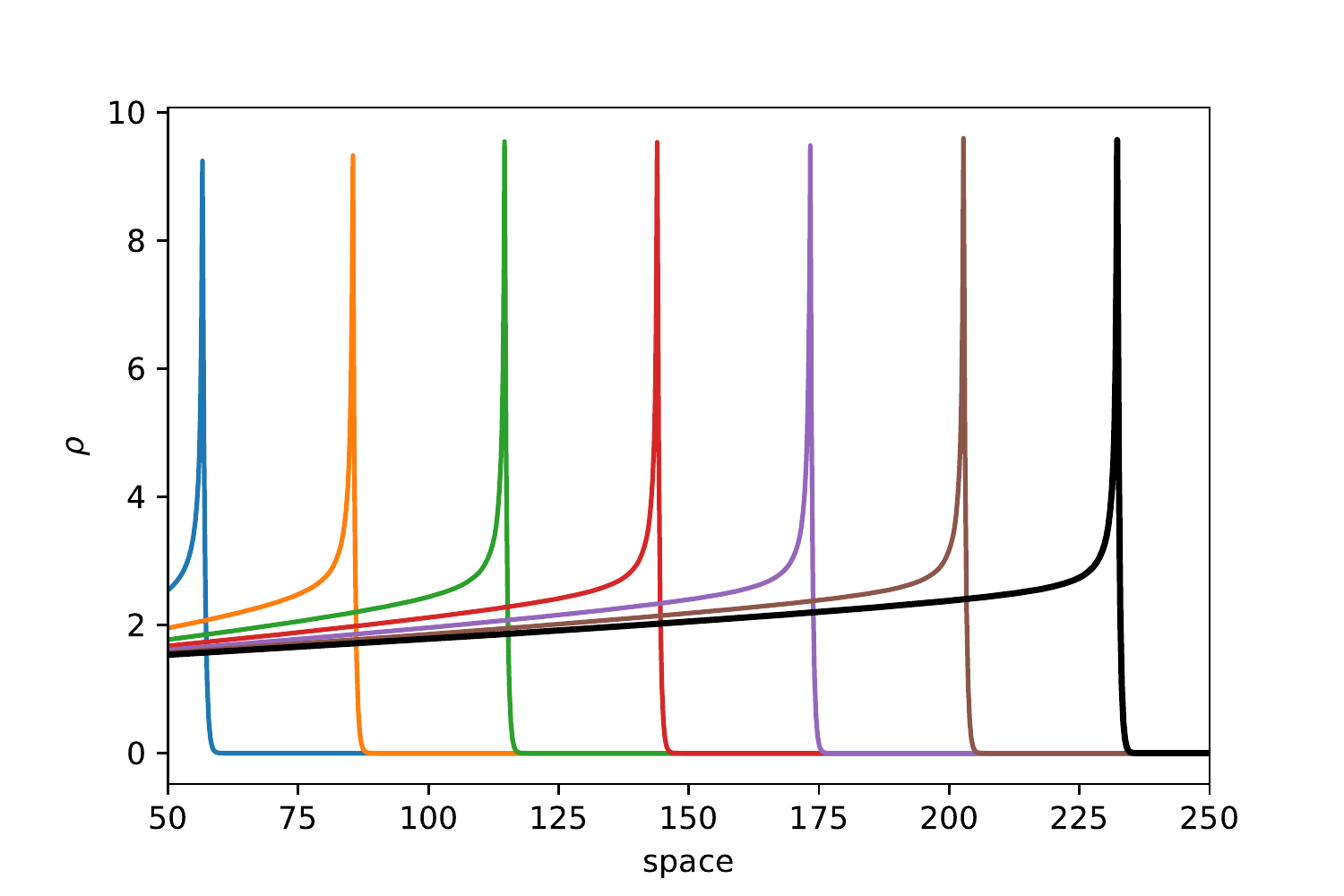}(b)\\
\includegraphics[width=.66\linewidth]{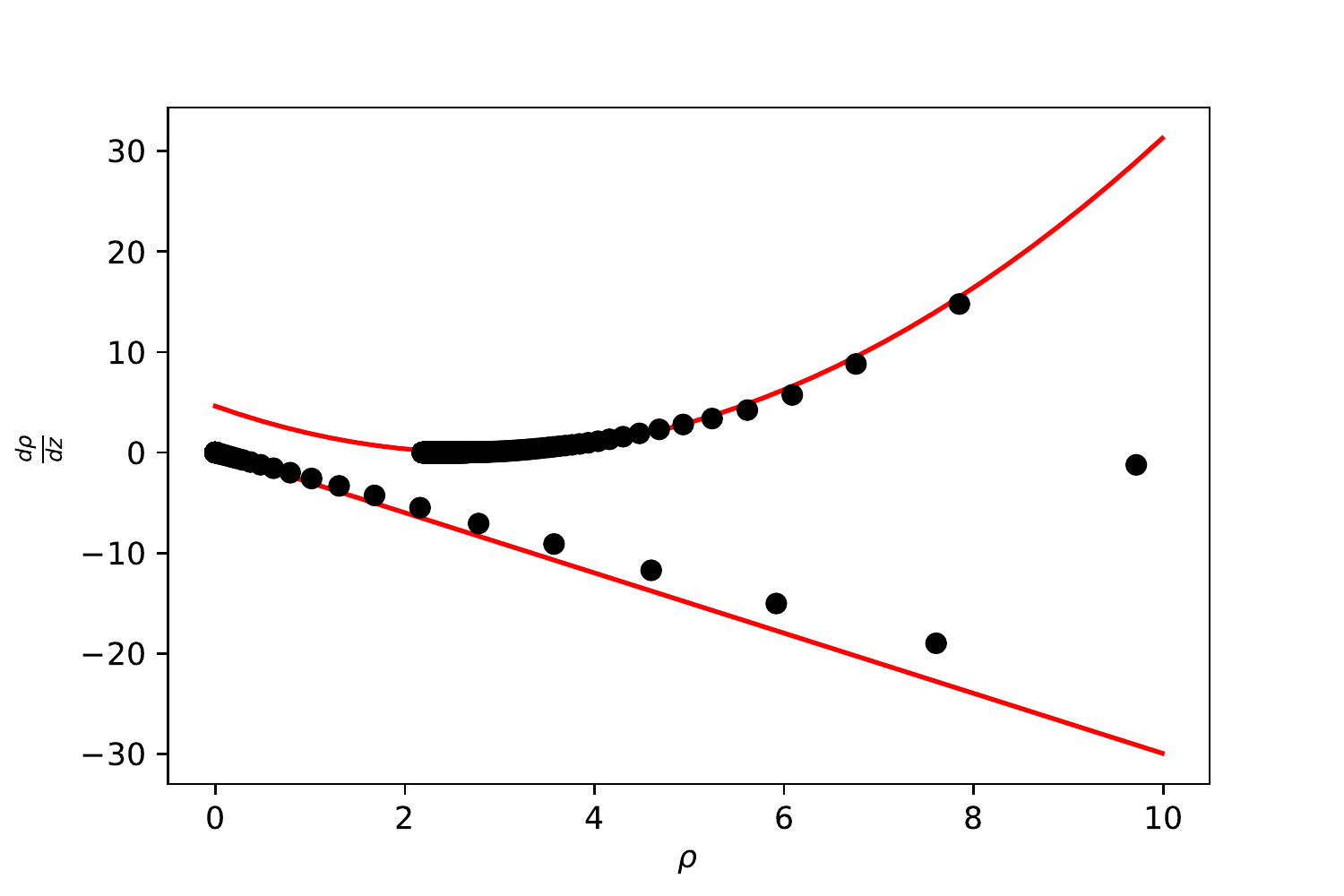}(c)\\
\caption{(a) Traveling wave propagation obtained after long time simulations of the Cauchy-problem \eqref{eq:extended KS growth}--\eqref{eq:gogrowlog} with parameters $(d=1,\chi = 2, r = 1, D = 0, k = 1, S_\init = 8, S_0 = 2$). (b) The density profile is shown at successive times in the moving frame. Note the low decay at the back of the wave, which is the signature of singular point in the ODE \eqref{eq:ODE z-} together with the choice of $J$ that cancels the discriminant in \eqref{eq:sign J}. The numerical speed is $c_\mathrm{num}\approx 3.17$, close to the theoretical one, $2\sqrt{\log(4)} \approx 3.33$. (c) To better assess our Claim  \ref{claim}, the numerical solution is plotted in the phase plane $(\rho,\rho')$ (black dots), against the theoretical curves, that is $\rho' = -\lambda \rho$ (for $z>0$), and $\rho' = \frac{k\chi}{cd}\left (\rho - \frac{c^2}{2k\chi}\right )^2$ \eqref{eq:ODE z-} (red lines). The isolated point on the right corresponds to the transition at $z=0$, where the expected theoretical profile has a  $\mathcal C^1$ discontinuity. We believe that the discrepancy is  due to numerical errors.  
\label{fig:roxana0}}
\end{center}
\end{figure}

\begin{figure}
\begin{center}
\includegraphics[width=.66\linewidth]{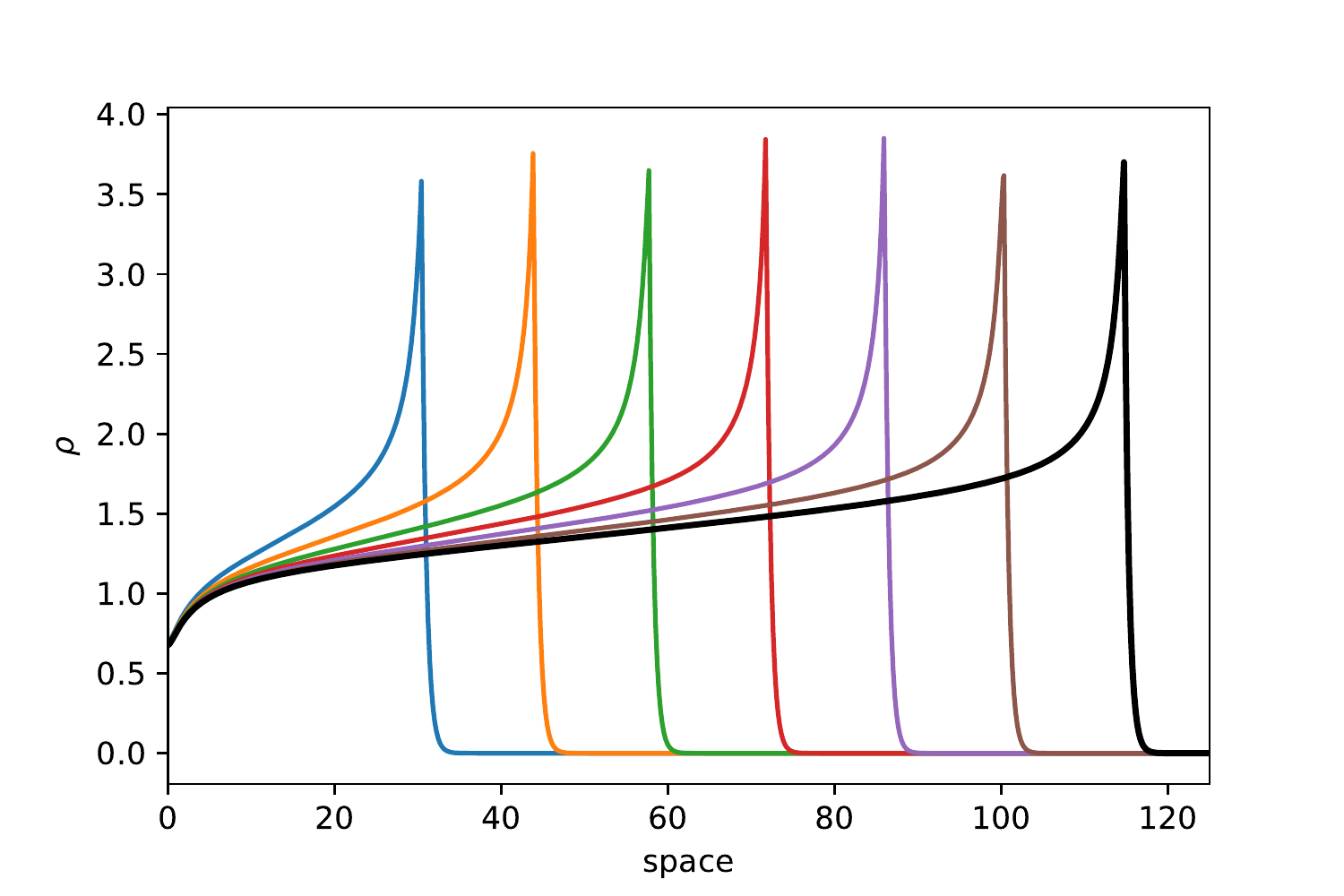}(a)\\
\includegraphics[width=.66\linewidth]{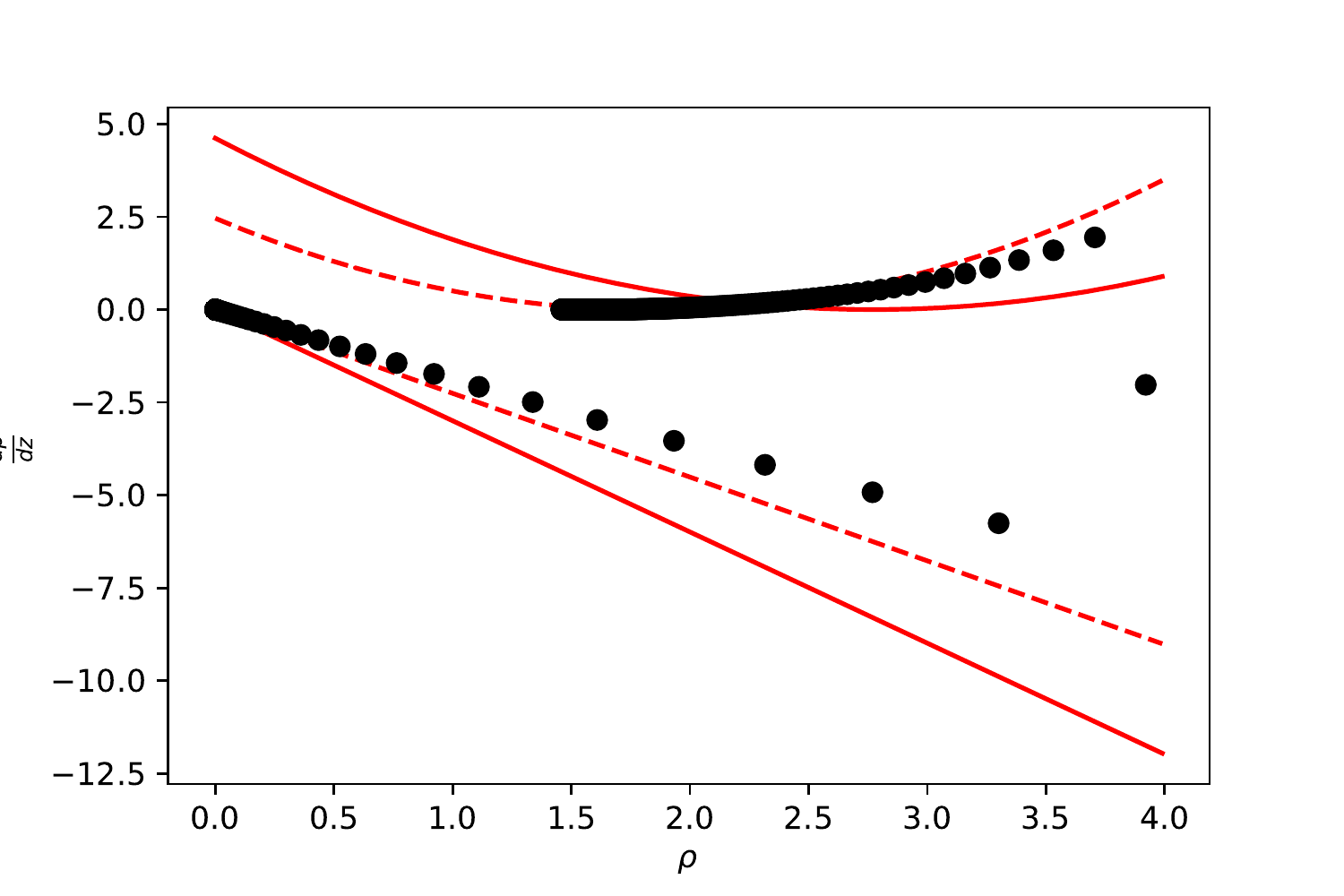}(b)
\caption{Same as in Figure \ref{fig:roxana0}, except for the diffusion coefficient of the chemical which is set to $D =1$. (a) We observe propagation of a traveling wave in the long time asymtptotic with a reduced speed. Clearly, the wave profile differs significantly from \ref{fig:roxana0}b. (b) In particular, the solution in the phase plane does not align with the theoretical expectation available in the case $D=0$ (red plain curves). It aligns much better with the theoretical expectation computed from the equations \eqref{eq:log S}--\eqref{eq:add restr} taking the reduced numerical speed as an input (red dashed curves). We believe that the discrepancy is  due to numerical errors.\label{fig:roxana1}}
\end{center}
\end{figure}

At this point, we conjecture that the minimal speed \eqref{eq:roxana} giving rise to admissible solutions is selected when the Cauchy problem is initiated with localized initial data. 

\begin{claim}\label{claim}
Starting from a compactly supported initial data, the asymptotic spreading speed of solutions to \eqref{eq:extended KS growth}--\eqref{eq:gogrowlog} is given by \eqref{eq:roxana}. 
\end{claim}

This claim is supported by numerical exploration of the system in some range of parameters, see Figure \ref{fig:roxana0} for one typical set of paramaters. On the one hand, the claim is not surprising in the case of small bias, when $c = 2\sqrt{rd}$. In fact, this corresponds to the standard mechanism of speed determination at the edge of the front in reaction-diffusion equation with pulled waves. This was indeed confirmed in the previous model \eqref{eq:extended KS growth}--\eqref{eq:gogrow} \cite{cochet-escartin_hypoxia_2021,demircigil_notitle_nodate}. On the other hand, we emphasize that it does look surprising in the case of large bias, when $c = 2\sqrt{ r\chi \log \left ( \frac{S_\init}{S_0} \right )}$. In the latter case, the selection of the minimal speed would come from a discriminant condition {\em at the back of the wave}, which would be a quite original phenomenon, up to our knowledge.

\section{Conclusion and perspectives}

We exposed the original contribution of Keller and Segel devoted to chemotactic waves of bacteria, and discussed its limitations. These limitations are mainly concerned with the possible lack of positivity of the chemical concentration in the model. A pair of extensions were described. They both resolve the positivity issue, while keeping analytical solvability of the waves thanks to the specific choice of piecewise homogeneous models. In addition, they are both supported by biological experiments, respectively with bacteria {\em E. coli} and Dicty cells. \bigskip

To conclude, let us mention some open problems, either on the mathematical or on the modeling side. 

\subsubsection*{Determinacy of the speed at the back of the wave.} The result stated in Theorem \ref{th:roxana} appeared quite unexpectedly. If further numerical exploration with alternative schemes tends to confirm our Claim \ref{claim}, we believe that understanding the mechanism of speed selection is an interesting, and possibly original problem {\em per se}. We stress out that this mechanism occurs at $z = -\infty$, in the sense that the sign condition on the discriminant in \eqref{eq:ODE z-} ensures that  the cell density remains positive for negative $z$. Alternatively speaking, we face a situation which is the mirror of the standard mechanism of speed determinacy at $z = +\infty$ in the F/KPP equation. 

\subsubsection*{Traveling waves with non-zero chemical diffusion.} Figure \ref{fig:roxana1} shows the numerical simulation  of the Cauchy problem \eqref{eq:extended KS growth}--\eqref{eq:gogrowlog} with a chemical diffusion coefficient $D$ of order one. It seems that the solution converges towards a traveling wave profile as $t\to +\infty$ with reduced speed as compared to the case without chemical diffusion (Figure \ref{fig:roxana0}). Moreover, the numerical wave plotted in the  phase plane shows a similar pattern (compare Figure \ref{fig:roxana0}c and \ref{fig:roxana1}b), suggesting similar mechanisms occurring at $z=-\infty$ (in particular, a vanishing discriminant in the super-critical case $c> 2\sqrt{rd}$). However, since the relationship \eqref{eq:log S} is not satisfied with non-zero diffusion, we are lacking one equation to perform explicit computations. There exist multiple works extending the construction of waves for the original model \eqref{eq:KS} to the case of non-zero chemical diffusion. This may give some hints to address this question. 

\subsubsection*{Stability.}
Although stability in the Keller-Segel model \eqref{eq:KS} has drawn some attention, with a nearly complete picture by now, stability of the traveling wave solutions to the models presented in Sections \ref{sec:scenario 1} and \ref{sec:scenario 2} is almost entirely open.  The first author and Hoffmann proved local non-linear stability of standing waves for \eqref{eq:saragosti} (without the SGG signaling $S$), assuming that the attractant concentration $A$ is quasi-stationary (solving an elliptic equation at any time). They performed a change of coordinates to by-pass the discontinuity of the advection coefficient, and used higher-order energy methods to handle the singular term of the coupling.  

Nevertheless, numerical investigation performed at the occasion of this work, with simple finite volume, semi-implicit, upwind schemes, argue in favor of stability of all the waves described in \ref{sec:scenario 1} and \ref{sec:scenario 2}.

\subsubsection*{Spatial sorting.} Another open problem is the theoretical analysis of spatial sorting in bacteria collectives when the individuals have different chemotactic sensitivities. In \cite{fu_spatial_2018}, remarkable experiments on bacteria {\em E. coli}, together with a very elegant analytical argument, indicated that cells can move together despite their differences. The argument of \cite{fu_spatial_2018} goes as follows: assume that there exist multiple types of bacteria consuming a single nutrient $S$, and that each type is characterized by a chemotactic sensivity $\chi_i$  ;  suppose that, for each type, the chemotactic advection is of the form $\bchi_i\left (S,\frac{\partial S}{\partial x}\right ) = \chi_i \frac{\partial F(S)}{\partial x}$, say the logarithmic gradient as in the original model \eqref{eq:KS} ; suppose that the solution of each type converges towards a traveling  wave in the long-time, with a common speed $c$, so that the flux is asymptotically zero in the moving frame for each type: 
\begin{equation}\label{eq:Fu}
(\forall i)\quad - c - d  \frac{d}{d z}\left (  \log \rho_i\right ) +  \chi_i \frac{d }{d z} F(S) = 0\,.
\end{equation} 
Evaluating \eqref{eq:Fu} at  the maximum point of the density  $\rho_i$, say $z_i^*$, we would get that 
\begin{equation}\label{eq:Fu1}
c =  \chi_i \frac{d }{d z} F(S)(z_i^*)\,.
\end{equation}
Differentiating \eqref{eq:Fu} at $z = z^*_i$, it could be deduced that 
\begin{equation}\label{eq:Fu2}
\frac{d^2}{dz^2}F(S)(z_i^*) =  d  \frac{d^2 }{d z^2}\left ( \log \rho_i\right )(z_i^*) \leq 0 \,.
\end{equation}
The combination of \eqref{eq:Fu1} and \eqref{eq:Fu2} says that the peaks $(z^*_i)$ of the densities $(\rho_i)$ which are traveling together are restricted to the interval where $F(S)$ is concave. Moreover, they are ordered in such a way that $(\chi_i < \chi_j) \Rightarrow (z_i^* < z_j^*)$. This nice calculation indicates that different phenotypes could migrate collectively despite their differences. The intuitive reason, which can be read on \eqref{eq:Fu1}, is that larger chemosensitivity $\chi_i$ naturally pushes the cells ahead, where they experience shallower gradients. Nonetheless, the analysis in \cite{fu_spatial_2018} is not complete, as the existence of a stable traveling waves of different types with a common speed  is taken for granted. 

There exist previous theoretical works about collective migration of different phenotypes within the same chemical environment. We refer for instance to \cite{lin_development_2014}, which adopted the framework of the original model by Keller and Segel \eqref{eq:KS}. In view of the discussion above, the stability of their theoretical outcomes is questionable. In \cite{emako_traveling_2016}, the authors extend the framework of Section \ref{sec:scenario 1}, including two subpopulations with different chemotactic phenotypes. This work was supported by experimental data. However, the discussion in \cite{fu_spatial_2018} makes it clear that the framework of \cite{emako_traveling_2016} is not directly compatible with their findings. Actually, it is one consequence of the advection speed discontinuity in \eqref{eq:saragosti} that the maximum peak density is located at the sign transition, whatever the chemosensitivity coefficient is, hence violating the nice relationship \eqref{eq:Fu1}. 

Preliminary investigations suggest that the framework of Section  \ref{sec:scenario 2} cannot be readily extended as well. Indeed, signal-dependent growth counter-balances the fact that more efficient chemotactic types experience shallower gradients, because they have better access to nutrient. This triggers natural selection of the more efficient type by differential growth (results not shown). 

To our knowledge, there is no clear mathematical framework  to handle the remarkable experiments and biological insights as shown in \cite{fu_spatial_2018}, at the present time.  

\printbibliography

\end{document}